\documentclass[12pt]{article}
\usepackage{amsxtra,amssymb,amsmath,enumerate,graphicx,bbm,hyperref,amsthm,color}
\usepackage[active]{srcltx}
\usepackage{t1enc}
\usepackage{aurical}
\usepackage[T1]{fontenc}
\usepackage{float}

\newtheorem{lemma}{Lemma}[section]
\newtheorem{proposition}[lemma]{Proposition}
\newtheorem{theorem}[lemma]{Theorem}

\newtheorem{remark}[lemma]{Remark}

\newcommand{\e}{\mathbb{E}}
\newcommand{\E}{\mathbb{E}}

\renewcommand{\P}{\mathbb{P}}
\newcommand{\R}{\mathbb{R}}

\newcommand{\1}{{\mathbf 1}}

\def\ti{{\tau_{i}}}
\def\tip{{\tau_{i+1}}}
\def \Esp#1{{\mathbb E}\left[#1\right]}

\def\be{\begin{align}}
\def\ee{\end{align}}
\def\b*{\begin{eqnarray*}}
\def\e*{\end{eqnarray*}}

\def\vp{\varphi}

\def\be{\begin{eqnarray}}
\def\ee{\end{eqnarray}}
\def\beq{\begin{equation}}
\def\eeq{\end{equation}}
\def\b*{\begin{eqnarray*}}
\def\e*{\end{eqnarray*}}
\def\bi{\begin{itemize}}
\def\ei{\end{itemize}}

\def \1{{\bf 1}}
\def\vp{\varphi}
\def\eps{\varepsilon}

\def\={\;=\;}

\def\x{\times}

\def\Esp#1{\mathbb{E}\left[#1\right]}
\def\Pro#1{\mathbb{P}\left[{#1}\right]}

\def \proof{{\noindent \bf Proof. }}
\def \ep{\hbox{ }\hfill$\Box$}
 \def\reff#1{{\rm(\ref{#1})}}

\def\Pas{\mathbb{P}-\mbox{a.s.}}

 \def\vs#1{\vspace{#1mm}}

\def\ti{{t_i}}
\def\tip{ {t_{i+1}} }

\def \E{\mathbb{E}}
\def \F{\mathbb{F}}

\def \P{\mathbb{P}}
\def \Q{\mathbb{Q}}
\def \R{\mathbb{R}}

\def\Ac{{\cal A}}

\def\Fc{{\cal F}}
\def\Gc{{\cal G}}

\def\Pc{{\cal P}}
\def\Sc{{\cal S}}

\makeatletter \@addtoreset{equation}{section}

\@addtoreset{Definition}{section}

\@addtoreset{Theorem}{section}

\@addtoreset{Proposition}{section}

\@addtoreset{Assumption}{section}

\@addtoreset{Corollary}{section}

\@addtoreset{Lemma}{section}

\@addtoreset{Remark}{section}

\@addtoreset{Example}{section}

  \def\ti{t_{i}^{n}}
\def\tip{t_{i+1}^{n}}

\def\Yn{Y^{n }}
\def\deltan{\delta^{n }}
\def\Vn{V^{n }}
\def\Sn{X^{n }}

\def\vrm{{\rm v}}
\def\vrmg{{\rm v}_{\bar \gamma}}
\def\Gcg{\Gc_{\bar \gamma}}
\def\vrmgbed{\bar {\rm v}^{\epsilon,K,\delta}_{\bar \gamma}}

\begin{document}

\title{Hedging of covered options with linear market impact and gamma constraint}

\author{B. Bouchard\thanks{Paris-Dauphine University, PSL Research University, CEREMADE, UMR 7534, 75775 Paris cedex 16, France.} \thanks{Research supported by ANR Liquirisk}, G. Loeper\thanks{Monash University, School of Mathematical Sciences, Victoria 3800
Australia.} and Y. Zou\footnotemark[1]}
\maketitle

\begin{abstract}
Within a financial model with  linear price impact, we study the problem of hedging a covered European option under gamma constraint. Using stochastic target and  {partial differential equation} smoothing techniques, we prove that the super-replication price is  {the}  viscosity solution of a fully non-linear parabolic equation.  As a by-product, we show how $\eps$-optimal strategies can be constructed.  {Finally, a numerical resolution scheme is proposed.}
\end{abstract}

\vspace{1em}

\section*{Introduction}   

Inspired by \cite{AbergelLoeper,loeper2013option},  {the authors in} \cite{bouchard2015almost} considered a financial market with permanent price impact, in which the impact function behaves as a linear function (around the origin) in the number of bought stocks.  This class of models is dedicated to the pricing and hedging of derivatives in situations  where the notional of the product hedged is such that the delta-hedging is non-negligible compared to the average daily volume traded on the underlying asset. Hence, the delta-hedging strategy  will have an impact on the price dynamics, and  will also incur liquidity costs. The linear impact models studied in \cite{AbergelLoeper,bouchard2015almost,loeper2013option}  incorporate both the effects into the pricing and hedging of the derivative, while maintaining the completeness of the market (up to a certain extent) and eventually leading to exact replication strategies.
As in perfect market models, this approach can provide an approximation of the real market conditions that can be used by practitioners to help them in the design of a suitable hedge in a systematic way,  without having to rely on an ad hoc risk criterion.

In  \cite{bouchard2015almost}, the authors considered the hedging of a cash-settled European option: at inception the option seller has to build the initial delta-hedge, and conversely at maturity the hedge must be liquidated to settle the final claim in cash.
It is shown therein that the price function of the optimal super-replicating strategy no longer solves a linear parabolic equation, as in the classical case, rather a quasi-linear one. Moreover, the hedging strategy consists in following a modified delta-hedging rule where the delta is computed at the  ``unperturbed'' value of the underlying, i.e.,~the one the underlying would have if the trader's position were liquidated immediately.

The approach and the results obtained in \cite{bouchard2015almost} thus differ substantially from \cite{AbergelLoeper,loeper2013option}. Indeed, while in \cite{AbergelLoeper,loeper2013option}   the impact model considered is the same, the control problem is different in the sense that it is applied to the hedging of \emph{covered options}. This refers to situations where the buyer of the option delivers at inception the required initial delta position, and accepts a mix of stocks (at their current market price) and cash as payment of the final claim, which eliminates the cost incurred by the initial and final hedge. Quite surprisingly, this is not a genuine  approximation of the problem studied in  \cite{bouchard2015almost}.  The question of the initial and final hedge is fundamental, to the point that the structure of the pricing question is completely different: in  \cite{bouchard2015almost} the equation is quasi-linear, while it is fully non-linear in  \cite{AbergelLoeper,loeper2013option}.

As opposed to \cite{bouchard2015almost}, authors in \cite{AbergelLoeper,loeper2013option} use a verification argument to build an exact replication strategy. Due to the special form of the non-linearity, the equation is ill-posed when the solution does not satisfy a gamma-type constraint. The aim of the current paper is to provide a direct characterization via stochastic target techniques, and to incorporate right from the beginning a gamma constraint on the hedging strategy.

 The super-solution property can be proved by (essentially) following  the arguments of \cite{cheridito2005multi}. The sub-solution characterization is much more difficult to obtain. Actually, we could not prove the required geometric dynamic programming principle, unlike \cite{cheridito2005multi}, because of the strong interaction between the hedging strategy and the underlying price process due to the market impact.  Instead,  we use the smoothing technique  developed in \cite{bouchard2013stochastic}. We construct a sequence of smooth super-solutions   which, by a verification argument, provide upper-bounds for  the super-hedging price. As they converge to a solution of the targeted pricing equation,  a comparison principle argument implies that their limit is the super-hedging price. As a by-product, this construction provides explicit $\eps$-optimal hedging strategies.  We also provide a comparison principle and a numerical resolution scheme. For simplicity, we first consider a model  that has only permanent price impact, we explain in Section \ref{sec: resilience} why adding a resilience effect does not affect our analysis. Note that this is because the resilience effect we consider here has no quadratic variation, as opposed to \cite{AbergelLoeper}, in which the resilience can destroy the parabolicity of the equation, and renders the exact replication non optimal.

We close this introduction by pointing out some related references. \cite{CetinJarrowProtter}  incorporates liquidity costs but no price impact, the price curve is not affected by the trading strategy.  It can be modified by adding restrictions on admissible strategies as in \cite{CetinSonerTouzi} and  \cite{SonTouzDyn}. This  leads to  a modified pricing equation, which exhibits a quadratic term in the second order derivative of the solution, and renders the pricing equation fully non-linear, and even not unconditionally parabolic. Other articles  focus  on the derivation of the price dynamics through clearing condition, see e.g.,~\cite{Frey}, \cite{Sircar}, \cite{Schon} in which the supply and demand curves   arise from ``reference'' and ``program'' traders (i.e., option hedgers); leading to a modified price dynamics, but without taking into account the liquidity costs, see also \cite{Liu}.   Finally,    the series of papers \cite{SonTouz}, \cite{cheridito2005multi}, \cite{SonTouzDyn} addresses the liquidity issue indirectly by imposing  bounds on the ``gamma'' of admissible trading strategies, no liquidity cost or price impact are modeled explicitly.

\vspace{5mm}

\noindent{\bf General notations.}
Throughout this paper,  $\Omega$ is the canonical space of continuous functions on $\R_{+}$ starting at $0$, $\P$ is the Wiener measure, $W$ is the canonical process, and $\F=(\Fc_{t})_{t\ge 0}$ is {the augmentation of its} raw filtration $\F^{\circ}=(\Fc^{\circ}_{t})_{t\ge 0}$.  All random variables are defined on $(\Omega,\Fc_{\infty},\P)$. We denote by $|x|$   the Euclidean norm of $x\in \R^{n}$, the integer  $n\ge 1$ is given by the context. {Unless otherwise specified, inequalities involving random variables are taken in the $\Pas$ sense. We use the convention $x/0={\rm sign}(x)\times \infty$ with ${\rm sign}(0)= +$. }

\section{Model and hedging problem}\label{sec: dynamics}

This section is dedicated to the derivation of the dynamics and the description of the gamma constraint. We also explain in detail how the pricing equation can be obtained and state our main result.

\subsection{Impact rule and discrete time trading dynamics}\label{sec: impact rule and discrete time trading}

We consider the framework studied in \cite{bouchard2015almost}. Namely, the impact of a strategy on the price process is modeled by an impact function $f$:
 the price variation due to buying a (infinitesimal)   number $\delta\in \R$ of shares is $\delta f(x)$,  given that the price of the asset is $x$ before the trade. The cost of buying the additional  $\delta$ units is
$$
\delta x+\frac12 \delta^{2} f(x)=\delta \int_{0}^{\delta}\frac1\delta (x+\iota f(x) ) d\iota,
$$
in which
$$
 \int_{0}^{\delta}\frac1\delta (x+\iota f(x) ) d\iota
$$
can be interpreted as the average cost for each additional unit.

Between two trading instances $\tau_{1}, \tau_{2}$ with $\tau_{1}\le \tau_{2}$, the dynamics of the stock is given by the strong solution of the stochastic differential equation
$$
dX_{t}=\mu(X_{t})dt +\sigma(X_{t}) dW_{t}.
$$

Throughout this paper, we assume that
\begin{equation}\label{eq: H1}
\begin{array}{c}
\mbox{${f\in C^{2}_{b}}$   and $\inf f>0$,} \\
\mbox{ $(\mu,\sigma)$ is Lipschitz and bounded, {$ \inf\sigma>  0$}. }
  \end{array}
\end{equation}

As in \cite{bouchard2015almost}, the number of shares the  trader would like to hold is given by a continuous It\^{o} process  $Y$ of the form
\be\label{eq: def Y sans saut}
Y =Y_{0}+\int_{0}^{\cdot} b_{s} ds +\int_{0}^{\cdot}a_{s}dW_{s}.
\ee
We say\footnote{In \cite{bouchard2015almost}, $(a,b)$ is only required to be progressively measurable and essentially bounded. The additional restrictions imposed here will be necessary for our results in Section \ref{sec: weak form}.} that $(a,b)$ belongs to ${\Ac^{\circ}_{k}}$ if $(a,b)$ is  continuous, $\F$-adapted,
$$
a=a_{0}+\int_{0}^{\cdot}\beta_{s} ds +\int_{0}^{\cdot} \alpha_{s} dW_{s}
$$
where $(\alpha,\beta)$ is  continuous, $\F$-adapted,  and $\zeta:=(a,b,\alpha,\beta)$ is essentially bounded by $k$  and such that
\b*
\E\left[\sup\left\{|\zeta_{s'}-\zeta_{s}|,\; t\le s \le s'\le s+\delta\le T\right\}|\Fc_{t}^{\circ} \right]\le k\delta
\e*
for all $0\le \delta\le 1$ and $t\in [0,T-\delta]$.

We then define 
\b*
& {\Ac^{\circ}:=\cup_{k}\Ac^{\circ}_{k}}.&
\e*
To derive the continuous time dynamics, we first consider a discrete time setting and then pass to the limit. In the discrete time setting, the position is re-balanced only at times
$$
\ti:=iT/n,\;i=0,\ldots,n, \;n\ge 1.
$$
In other words, the trader keeps the position   $Y_{\ti}$ in stocks over each time interval $[\ti,\tip)$. Hence, his position in stocks at $t$ is
\be\label{eq: def Yn sans saut}
\Yn_{t}:=\sum_{i=0}^{n-1} Y_{\ti}\1_{\{\ti\le t<\tip\}}+Y_{T}\1_{\{t=T\}},
\ee
 and the number of  shares purchased at $\tip$ is
$$
\deltan_{\tip}:=  Y_{\tip}-Y_{\ti} .
$$
Given our impact rule, the corresponding dynamics for the stock price  process   is
\be\label{eq: dyna Sn}
\Sn  = X_{0} + \int_{0}^{\cdot}\mu(\Sn_{s})ds +\int_{0}^{\cdot} \sigma(\Sn_{s}) dW_{s} + \sum_{i=1}^{n}\1_{[\ti,T]} \deltan_{\ti} f(\Sn_{\ti-}),
\ee
in which $X_{0}$ is a constant.

The portfolio process is described as the sum $\Vn$ of the amount of cash held and the potential wealth $\Yn \Sn$ associated to the position in stocks:
\b*
V^{n}=\mbox{cash position } + \Yn \Sn.
\e*
It does not correspond to the liquidation value of the portfolio, except when $\Yn=0$. This is due to the fact that the liquidation of $\Yn$ stocks    does not generate a gain equal to $\Yn \Sn$, because of the price impact.
However, one can infer  the exact composition in cash and stocks of the portfolio from the knowledge of the couple $(\Vn,\Yn)$.

Throughout this paper, we assume that the risk-free interest rate is zero (for ease of notations). Then,
\begin{equation}
\Vn = V_{0 } + \int_{0}^{\cdot}\Yn_{s-}d\Sn_{s} + \sum_{i=1}^{n}\1_{[\ti,T]}  \frac12 ( \deltan_{\ti})^{2}f(\Sn_{\ti-}). \label{eq: dyna Vn}
\end{equation}
This wealth equation is derived as in \cite{bouchard2015almost} following elementary calculations. The last term of the right-hand side comes from the fact that, at time $\ti$, $\deltan_{\ti}$ shares are bought at the average execution price $\Sn_{\ti-} + \frac{1}{2}\deltan_{\ti}f(\Sn_{\ti-})$,
and the stock's price ends at $\Sn_{\ti-}+\deltan_{\ti}f(\Sn_{\ti-})$, whence the additional profit term. However, one can check that a profitable round trip trade can not be built, { see \cite[Remark 3]{bouchard2015almost}.}

 \begin{remark} Note that in this work we restrict ourselves to a permanent price impact, no resilience effect is modeled. We shall explain in Section \ref{sec: resilience}
 below why taking resilience into account does not affect our analysis. See in particular Proposition \ref{prop: resilience}.
 \end{remark}

\subsection{Continuous time trading dynamics}

The continuous time trading dynamics is obtained by passing to the limit $n\to \infty$, i.e.,~by considering strategies with  increasing frequency of rebalancement.
\begin{proposition}{\cite[Proposition 1]{bouchard2015almost}}\label{prop : lim quand n to infty} Let $Z:=(X,Y,V)$ where $Y$ is defined as in \reff{eq: def Y sans saut} for some $(a,b)\in \Ac^{\circ}$, and  $(X,V)$ solves
\begin{align}
X&=X_{0 }+\int_{0}^{\cdot} \sigma(X_{s}) dW_{s} + \int_{0}^{\cdot} f(X_{s}) dY_{s}+\int_{0}^{\cdot} (\mu(X_{s})+a_{s}(\sigma f')(X_{s}) )ds\nonumber\\
&=X_{0 }+\int_{0}^{\cdot} \sigma_{X}^{a_{s}}(X_{s}) dW_{s} + \int_{0}^{\cdot} \mu_{X}^{a_{s},b_{s}}(X_{s})ds  \label{eq: S lim conti}
\end{align}
with
\begin{align*}
\sigma_{X}^{a_{s}}:=(\sigma+a_{s}f)\:\;,\;\;\mu_{X}^{a_{s},{b_{s}}}:=( {\mu+}b_{s}f+a_{s}\sigma f' ),
\end{align*}
and
\begin{equation}
V=V_{0 }+ \int_{0}^{\cdot} Y_{s}dX_{s}+\frac12 \int_{0}^{\cdot} a^{2}_{s} f(X_{s}) ds.\label{eq: V lim conti}
\end{equation}
Let $Z^{n}:=(X^{n},Y^{n},V^{n})$ be defined as in \reff{eq: dyna Sn}-\reff{eq: def Yn sans saut}-\reff{eq: dyna Vn}.
Then,  there exists a constant $C>0$ such that
\b*
\sup_{[0,T]}\Esp{  |Z^{n} -Z |^{2} }\le Cn^{-1}
\e*
for all $n\ge 1$.
 \end{proposition}

 For the rest of the paper, we shall therefore consider \reff{eq: V lim conti}-\reff{eq: S lim conti} for the dynamics of the portfolio and price processes.

 \begin{remark} As explained in  \cite{bouchard2015almost},  the previous analysis could be extended to a non-linear impact rule   in   the size of the order.  To this end, we note that the continuous time trading dynamics described above would be the same  for a more general impact rule $\delta \mapsto F(x,\delta)$ whenever it satisfies $F(x,0){=\partial^{2}_{\delta \delta}F(x,0)}=0$ and  $\partial_{\delta}F(x,0)=f(x)$. For our analysis, we only need to consider the value and the slope  of the impact function at the origin.
\end{remark}

\subsection{Hedging equation and gamma constraint}

 Given $\phi=(y,a,b)\in \R\x \Ac^{\circ}$ and $(t,x,v)\in [0,T]\x \R\x \R$,  we now write $(X^{t,x,\phi}$, $Y^{t,\phi}$ , $V^{t,x,v,\phi})$ for the solution of \reff{eq: S lim conti}-\reff{eq: def Y sans saut}-\reff{eq: V lim conti} associated to the control $(a,b)$ with time-$t$ initial condition $(x,y,v)$.

In this paper, we consider covered options, in the sense that the trader is given at the initial time $t$ the number of shares $Y_{t}=y$ required to launch his hedging strategy and can pay the option's payoff at $T$ in cash and stocks (evaluated at their time-$T$ value).  Therefore, he does not exert any immediate impact at time $t$ nor $T$ due to the initial building or final liquidation of his position in stocks. Recalling  that $V$ stands for the sum of the position in cash and the number of held shares multiplied by their price,  the super-hedging price at time $t$  of the option with payoff $g(X^{t,x,\phi}_{T})$ is defined as
 $$
   \vrm(t,x):=\inf\{ v=c+ yx~:~(c,y)\in \R^{2} \mbox{ s.t. }    \Gc(t,x,v,y)\ne \emptyset\},
 $$
 in which $  \Gc(t,x,v,y)$ is the set of elements $(a,b)\in \Ac^{\circ}$ such that $\phi:=(y,a,b)$ satisfies
 \begin{equation*}\label{eq: def condition de sur-rep}
 V^{t,x,v,\phi}_{T}\ge g(X^{t,x,\phi}_{T}).
 \end{equation*}

 In order to understand what the associated partial differential equation is, let us first rewrite the dynamics  of $Y$ in terms of $X$:
$$
dY^{t,\phi}_{t}=\gamma^{a_{t}}_{Y}(X^{t,x,\phi}_{t})dX^{t,x,\phi}_{t} +\mu^{a_{t},b_{t}}_{Y}(X^{t,x,\phi}_{t})dt
$$
with
\begin{align}\label{eq: def gamma a}
\gamma^{a}_{Y}  :=\frac{a}{\sigma +f a}\; \;\mbox{ and }\;\;
\mu^{a,b}_{Y}  :=b-{\gamma^{a}_{Y}\mu^{a,b}_{X}}.
\end{align}
Assuming  that the hedging strategy consists in tracking the super-hedging price, as in classical complete market models,
then one should have $V^{t,x,v,\phi}=\vrm(\cdot,X^{t,x,\phi})$. If $\vrm$ is smooth, recalling \reff{eq: S lim conti}-\reff{eq: V lim conti} and  applying It\^{o}'s lemma twice  implies
\begin{equation}\label{eq: Y=delta}
Y^{t,\phi}=\partial_{x}  \vrm(\cdot,X^{t,x,\phi}) \;\mbox{ , }\; \gamma^{a}_{Y}(X^{t,x,\phi})=\partial^{2}_{xx}  \vrm(\cdot,X^{t,x,\phi}),
\end{equation}
and
\begin{equation}\label{eq: dt =dt}
{\frac12} a^{2}f(X^{t,x,\phi})=\partial_{t}  \vrm(\cdot,X^{t,x,\phi}) + \frac12 {(\sigma_X^a)^{2}}(X^{t,x,\phi}) \partial^{2}_{xx}  \vrm(\cdot,X^{t,x,\phi}).
\end{equation}
{
Then,  the right-hand side of \reff{eq: Y=delta} combined with the definition of $\gamma^{a}_{Y}$ leads to
$$a=\frac{\sigma \partial^{2}_{xx}  \vrm(\cdot,X^{t,x,\phi})}{1-f\partial^{2}_{xx}  \vrm(\cdot,X^{t,x,\phi})} \;\mbox{ , }\;\sigma_X^a=\frac{\sigma}{1-f\partial^{2}_{xx}  \vrm(\cdot,X^{t,x,\phi})},$$
and  \reff{eq: dt =dt} simplifies to
\begin{equation}\label{eq: pricing pde sans contrainte}
\left[-\partial_{t}  \vrm - \frac12  \frac{\sigma^{2} }{(1- f\partial^{2}_{xx}  \vrm )}\partial^{2}_{xx}  \vrm\right](\cdot,X^{t,x,\phi}) =0 \;\;\mbox{ on } [t,T).
\end{equation}
 }
This is precisely the pricing equation obtained in \cite{AbergelLoeper,loeper2013option}.

Equation \reff{eq: pricing pde sans contrainte} needs to be considered with some precautions due to the singularity at $f\partial^{2}_{xx}  \vrm=1$. Hence, one needs to enforce that $1-f\partial^{2}_{xx}  \vrm$ does not change sign. We choose to restrict the solutions to satisfy
$1-f\partial^{2}_{xx}  \vrm >0$, since having the opposite inequality would imply that $a$ does not have the same sign as $\partial^{2}_{xx}  \vrm$, so that, having sold a convex payoff, one would sell when the stock goes up and buy when it goes down, a very counter-intuitive fact.



In the following, we impose that  {the constraint}
 \begin{equation}\label{eq: contrainte gamma}
{-k\le  } \gamma^{a}_{Y}(X^{t,x,\phi})\le \bar \gamma (X^{t,x,\phi})  \;,\;\mbox{ on } [t,T]\;\;\P-{\rm a.e.},
 \end{equation}
should hold for some $k\ge 0$, in which  $\bar \gamma$ is a bounded continuous map satisfying
 \begin{equation}\label{eq: diff bar gamma f}
  \iota\le \bar \gamma\le 1/f-\iota,\;\;\mbox{  for some $\iota>0$. }
 \end{equation}
{We denote by $\Ac_{k,\bar \gamma}(t,x)$ the collection of elements $(a,b)\in \Ac^{\circ}_{k}$ such that \reff{eq: contrainte gamma} holds and define
 $$
 \Ac_{\bar \gamma}(t,x):=\cup_{k\ge 0}\Ac_{k,\bar \gamma}(t,x).
 $$}

 Then, the equation \reff{eq: pricing pde sans contrainte} has to be modified to take this gamma constraint into account, leading naturally to
 \begin{equation}\label{eq: pde contrainte interior}
F[\vrmg]:=\min\left\{ -\partial_{t}\vrmg - \frac12  \frac{\sigma^{2} }{1- f\partial^{2}_{xx}\vrmg }\partial^{2}_{xx}\vrmg\;,\;\bar \gamma - \partial^{2}_{xx}\vrmg\right\}=0
\;\mbox{ on } [0,T)\x \R,
 \end{equation}
 in which $\vrmg$ is defined as $\vrm$ but with
 $$
 \Gcg(t,x,v,y):={\Gc(t,x,v,y)\cap \Ac_{\bar \gamma}(t,x)}
 $$
 in place of $\Gc(t,x,v,y)$. More precisely,
 \begin{equation}\label{eq: def vgamma}
   \vrmg(t,x):=\inf\{ v=c+ yx~:~(c,y)\in \R^{2} \mbox{ s.t. }    \Gcg(t,x,v,y)\ne \emptyset\}.
 \end{equation}

 As for the $T$-boundary condition, we know that $\vrmg(T,\cdot)=g$ by definition. However, as usual, the constraint on the gamma in \reff{eq: pde contrainte interior} should propagate up to the boundary and $g$ has to be replaced by its face-lifted version $\hat g$, defined as the smallest function above $g$ that is a viscosity super-solution of the equation $\bar \gamma - \partial^{2}_{xx}\vp\ge 0$. It is obtained by considering any twice continuously differentiable function $\bar \Gamma$ such that $\partial^{2}_{xx} \bar \Gamma=\bar \gamma$, and then setting
\begin{equation*}\label{eq: def hat g}
\hat{g}:=(g-\bar \Gamma)^{\text{conc}} +  \bar \Gamma,
\end{equation*}
in which the superscript \text{conc} means concave envelope, cf.~\cite[Lemma 3.1]{SonTouz}.\footnote{Obviously, adding an affine map to $\bar \Gamma$ does not change the definition of $\hat g$.}  Hence, we expect that
 \begin{equation*}
\vrmg(T-,\cdot)= \hat g
\;\mbox{ on }   \R.
 \end{equation*}
{From now on, we assume that}
\begin{align}\label{eq: hyp hat g}
\begin{array}{c}
\mbox{$\hat g$ is uniformly continuous,} \\
\mbox{{$g$ is lower-semicontinuous, $g^{-}$ is bounded and $g^{+}$ has linear growth.}}
\end{array}
\end{align}

We are now in a position to state our main result. From now on 
$$
\vrmg(T,x)\;  \mbox{ stands for } \lim_{\tiny \begin{array}{c}(t',x')\to (T,x)\\ t'<T\end{array}}  \vrmg(t',x')
$$ 
whenever it is well defined.

\begin{theorem}\label{thm: main} {The value function  $\vrmg$ is continuous with linear growth. Moreover,  $\vrmg$ is the unique   viscosity solution with linear growth} of
\be\label{eq: VisEqBd}
F[\vp]\1_{[0,T)}+(\vp-\hat{g})\1_{\{T\}}=0 \;\;\mbox{ on $[0,T]\x \R$.}
\ee
\end{theorem}

We conclude this section with additional remarks.
\begin{remark} {Note that $\hat g$ can be uniformly continuous without $g$ being continuous. Take for instance $g(x)=\1_{\{x\ge K\}}$ with $K\in \R$, and consider the case where $\bar \gamma>0$ is a constant. Then, $\hat g(x)=[ \1_{\{x\ge x_{o}\}} \frac{\bar \gamma}{2}(x-x_{o})^{2}]\wedge 1$ with $x_{o}:=K-(2/\bar \gamma)^{\frac12}$.}
\end{remark}

\begin{remark}\label{rem: hat g borne} The map $\hat g$ inherits the linear growth of $g$. Indeed, let $c_{0},c_{1}\ge 0$ be constants such that $|g(x)| \le w(x):=c_{0}+c_{1}|x|$. Since $\hat g\ge g$ by construction, we have $\hat g^{-}\le w$.  On the other hand, since $\bar \gamma\ge \iota>0$, by \reff{eq: diff bar gamma f}, it follows from the arguments in  \cite[Lemma 3.1]{SonTouz}  that   $\hat g\le (w-\tilde \Gamma)^{{\rm conc}}+\tilde \Gamma$, in which $\tilde \Gamma(x)=\iota x^{2}/2$.
Now, one can easily check by direct computations  that
$$
 (w-\tilde \Gamma)^{{\rm conc}} =(w-\tilde \Gamma)(x_{o})\1_{[-x_{o},x_{o}]} +(w-\tilde \Gamma)\1_{[-x_{o},x_{o}]^{c}}
$$
with $x_{o}:=c_{1}/\iota$. Hence, $ (w-\tilde \Gamma)^{{\rm conc}} +\tilde  \Gamma$ has the same  linear growth as $w$.
\end{remark}

\begin{remark} As will appear in the rest of our analysis, one could very well introduce a time dependence in the impact function $f$ and in $\bar\gamma$. Another interesting question studied by the
second author in  {\cite{loeper2013option} and \cite{loeper2014Solution}} concerns the smoothness of the solution and how the constraint on $\partial^2_{xx}\vrm$ gets naturally enforced by the fast diffusion arising when $1-f \partial^2_{xx}\vrm$ is close to 0.
\end{remark}

\begin{remark}[{Existence of a smooth solution to the original partial differential equation}]  When the pricing equation {\reff{eq: VisEqBd}} admits smooth solutions  {(cf.~\cite{loeper2013option} and \cite{loeper2014Solution})} that allow to use the verification theorem, then one can   construct exact replication strategies from the classical solution. By the comparison principle of Theorem \ref{thm: Comp} below, this shows that the value function is
the classical solution of the pricing equation, and that the optimal strategy exists and is an exact replication  strategy of the option with  payoff function $\hat g$. We will explain in Remark \ref{rem: almost optimal controls} below how almost optimal super-hedging strategies can be constructed explicitly even when no smooth solution exists.  \end{remark}

\begin{remark}[Monotonicity in the impact function]\label{rem: pde non decreasing} Note that the map $\lambda\in \R \mapsto \frac{\sigma^{2}(x)M}{1-\lambda M}$ is non-decreasing on $\{\lambda: \lambda M< 1\}$, for all $(t,x,M)\in [0,T]\x \R\x \R$. Let us now write $\vrmg$ as $\vrmg^{f}$ to emphasize its dependence on $f$, and  consider another {impact} function $\tilde f$ satisfying the same requirements as $f$. We denote by $\vrmg^{\tilde f}$ the corresponding super-hedging price. Then, the above considerations combined with Theorem \ref{thm: main}  and the comparison principle of Theorem \ref{thm: Comp} below imply that $\vrmg^{\tilde f}\ge \vrmg^{f}$
whenever $\tilde f\ge f$ on $\R$. The same implies that $\vrmg^{f}\ge {\rm v}$ in which ${\rm v}$ solves the heat-type  equation
\begin{equation*}
 -\partial_{t}\vp - \frac12  \sigma^{2}\partial^{2}_{xx}\vp=0
\;\mbox{ on } [0,T)\x \R,
 \end{equation*}
with terminal condition $\vp(T,\cdot)=g$ (recall that $\hat g\ge g$).  See Section \ref{sec: exemple num} for a numerical illustration of this fact.
\end{remark}


\section{Viscosity solution characterization}\label{sec: hedging}
 \def\Sc{{\cal S}}
\def\Srm{{\rm S}}
\def\Srmt{{\tilde \Srm}}
\def\vrmglk{\underline{\rm v}_{\bar \gamma}^{k}}
\def\vrmgl{\underline{\rm v}_{\bar \gamma}}
\def\vrmgh{\bar{\rm v}_{\bar \gamma}}
\def\vrmghe{\bar{\rm v}^{\epsilon,K}_{\bar \gamma}}
\def\vrmghed{\bar{\rm v}^{\epsilon,\delta}_{\bar \gamma}}

In this section, we provide the proof of Theorem \ref{thm: main}. Our strategy is the following.

\begin{enumerate}
\item First, we adapt the partial differential equation smoothing technique used in \cite{bouchard2013stochastic} to provide a smooth supersolutions $\vrmgbed$ of \reff{eq: VisEqBd} on $[\delta,T]\x \R$, with $\epsilon>0$,  from which super-hedging strategies can be constructed by a standard verification argument.  In particular, $\vrmgbed\ge \vrmg$ on $[\delta,T]\x \R$. Moreover,  this sequence has a uniform linear growth and converges to a viscosity  solution $\vrmgh$ of  \reff{eq: VisEqBd} as $\delta,\epsilon\to 0$ and $K\to \infty$. See Section \ref{sec: smoothing}.

\item Second, we construct a lower bound $\vrmgl$ for $\vrmg$ that is a supersolution of \reff{eq: VisEqBd}. It is obtained by considering a weak formulation of the super-hedging problem and following the arguments of  \cite[Section 5]{cheridito2005multi} based on one side of the geometric dynamic programming principle, see Section \ref{sec: weak form}.  It is shown that this function has linear growth as well.
\item We can then conclude by using the above and the comparison principle for \reff{eq: VisEqBd} of Theorem \ref{thm: Comp} below:   $\vrmgl\ge \vrmgh$ but $\vrmgl\le \vrmg\le \vrmgh$   so that $\vrmg=\vrmgh=\vrmgl$  and $\vrmg$ is a viscosity solution of \reff{eq: VisEqBd}, and has linear growth.
\item Our comparison principle, Theorem \ref{thm: Comp} below, allows us to conclude that $\vrmg$ is the unique  solution of \reff{eq: VisEqBd}  with linear growth.
\end{enumerate}

As already mentioned in the introduction,  unlike \cite{cheridito2005multi}, we could not prove the required geometric dynamic programming principle that should directly lead to a subsolution property (thus avoiding to use the smoothing technique mentioned in 1.~above).   This is due to  the strong interaction between the hedging strategy and the underlying price process through  the market impact.  Such a feedback effect  is not present in \cite{cheridito2005multi}.

\subsection{A sequence of smooth supersolutions}\label{sec: smoothing}

We first construct a sequence of smooth supersolutions $\vrmgbed$ of \reff{eq: VisEqBd}   which appears to be an upper bound for the super-hedging price $\rm v_{\bar \gamma}$, by a simple verification argument. For this, we adapt the methodology introduced in \cite{bouchard2013stochastic}: we first construct a viscosity solution of a version of \reff{eq: VisEqBd} with shaken coefficients (in the terminology of  \cite{krylov2000rate}) and then smooth it out with a kernel.
The main difficulty here is that our terminal condition $\hat g$ is unbounded, unlike \cite{bouchard2013stochastic}. This requires  additional non trivial technical developments.

\subsubsection{Construction of a solution for the operator with shaken coefficients}

We start with the construction of the operator with shaken coefficients. Given ${\epsilon}>0$ and a (uniformly) strictly positive continuous map $\kappa$ with linear growth, that will be defined later on,  let us  introduce a family of perturbations of {the operator appearing in} (\ref{eq: VisEqBd}):
$$
F_{\kappa}^{\epsilon}(t,x,q,M) := \min_{x'\in D^{\epsilon}_{\kappa}(x)} \min\left\{-q-\frac{\sigma^{2}(x')M}{2(1-f(x')M)}, \bar \gamma(x')-M\right\},
$$
where
\be\label{eq: def D eps}
D_{\kappa}^{\epsilon}(x):=\{x'\in \R: (x-x')/\kappa(x')\in [-\epsilon,\epsilon]\}.
\ee
For ease of notation, we set
$$
F_{\kappa}^{\epsilon}[\vp](t,x) := F_{\kappa}^{\epsilon}(t,x,\partial_{t}\vp(t,x), \partial^{2}_{xx}\vp(t,x)),
$$
whenever $\vp$ is smooth.

\begin{remark}\label{rem: Feps concave} For later use, note that  the map $M\in (-\infty,\bar \gamma(x)] \mapsto  \frac{\sigma^{2}(x)M}{2(1-f(x)M)}$ is non-decreasing and  convex, for each $x\in \R$, recall \reff{eq: diff bar gamma f}. Hence,  $(q,M)\in \R\x (-\infty,\bar \gamma(x)]\mapsto F_{\kappa}^{\epsilon}(\cdot,q,M)$ is concave and non-increasing in $M$, for all $\epsilon\ge 0$. This is fundamental for our smoothing approach to go through.
\end{remark}

We also  modify the original terminal condition $\hat{g}$  by using an approximating sequence whose elements are affine  for large values of $|x|$.

\begin{lemma}\label{lem: construction hat gK}  For all $K>0$ there exists a uniformly continuous map $\hat g_{K}$ and $x_{K}\ge K$   such that
\begin{itemize}
\item $\hat g_{K}$ is affine on $[x_{K},\infty)$ and on $(-\infty,-x_{K}]$
\item $\hat g_{K}=\hat g$ on $[-K,K]$
\item $\hat g_{K}\ge \hat g$
\item $\hat g_{K}-\bar \Gamma$ is concave for any $C^{2}$ function $\bar \Gamma$ satisfying $\partial^{2}_{xx} \bar \Gamma=\bar \gamma$.
\end{itemize}
Moreover, $(\hat g_{K})_{K>0}$ is uniformly bounded by  {a map with linear growth} and  converges to $\hat g$ uniformly on compact sets.
\end{lemma}

\proof Fix a $C^{2}$ function $\bar \Gamma^{\circ}$ satisfying $\partial^{2}_{xx} \bar \Gamma^{\circ}=\bar \gamma$. By definition, $\hat g-\bar \Gamma^{\circ}$ is concave.
Let us consider an element $  \Delta^{+}$ (resp.~$  \Delta^{-}$) of its super-differential at $K$ (resp.~$-K$). Set
\begin{align*}
\hat g^{\circ}_K(x):= & \hat g(x)\1_{[-K,K]}(x)\\
&+ \left[\hat g(K) +  (\Delta^{+}+\partial_{x}\bar \Gamma^{\circ}(K)) (x-K)\right]\1_{(K,\infty)}(x)\\
&+ \left[\hat g(-K) +  (\Delta^{-}+\partial_{x}\bar \Gamma^{\circ}(-K)) (x+K)\right]\1_{(-\infty,-K)}(x).
\end{align*}
Consider now another $C^{2}$ function $\bar \Gamma$ satisfying $\partial^{2}_{xx} \bar \Gamma=\bar \gamma$. Since  $\bar \Gamma^{\circ}$ and $\bar \Gamma$ differ only by an affine map, the concavity of $\hat g^{\circ}_{K}-\bar \Gamma$ is equivalent to that of $\hat g^{\circ}_{K}-\bar \Gamma^{\circ}$. The concavity of the latter follows from the definition of $\hat g^{\circ}_{K}$, as the superdiffential of $\hat g^{\circ}_{K}-\bar \Gamma^{\circ}$ is non-increasing by construction. In particular, $\hat g^{\circ}_K-\bar \Gamma^{\circ}\ge \hat g -\bar \Gamma^{\circ}$ and therefore $\hat g^{\circ}_{K}\ge \hat g$.
 \\
 We finally define $\hat g_{K}$ by
 \be\label{eq: def hat g_{K}}
 \hat g_{K}=\min\{\hat g^{\circ}_{K} ,(2c_{0}+c_{1}|\cdot|-\bar \Gamma^{\circ})^{\rm conc}+\bar \Gamma^{\circ}\},  
 \ee
 with $c_{0}>0$ and $c_{1}\ge 0$ such that
 $$
-c_{0}\le  \hat g(x)\le c_{0}+c_{1}|x|, \;x\in \R,
 $$
 recall Remark \ref{rem: hat g borne}. The function $\hat g_{K}$ has the same linear growth as $2c_{0}+c_{1}|\cdot|$, by the same reasoning as in Remark \ref{rem: hat g borne}. Since $2c_{0}>c_{0}$, $\hat g_{K}=\hat g^{\circ}_{K}=\hat g$ on $[-K,K]$. Furthermore, as the minimum of two concave functions is concave, so is $ \hat g_{K}-\bar \Gamma$   for any $C^{2}$ function $\bar \Gamma$ satisfying $\partial^{2}_{xx} \bar \Gamma=\bar \gamma$.
 The other assertions are immediate.
\ep
\\

We now set
\be\label{eq: def hat g K eps}
\hat g_{K}^{\epsilon}:=\hat g_{K}+\epsilon
\ee
and consider the equation
\be\label{eq: VisEqBdPert}
F_{\kappa}^{\epsilon}[\vp]\1_{[0,T)}+(\vp-\hat{g}_{K}^{\epsilon})\1_{\{T\}}=0.
\ee
We then choose $\kappa$ and $\epsilon_{\circ}\in (0,1)$ such that
\be\label{eq: cond kappa}
\begin{array}{c }
\mbox{$\kappa \in C^{\infty}$ with bounded derivatives of all orders,}\\
 \inf\kappa>0 \mbox{ and } \kappa=|\hat g_{K}|+1 \mbox{ on }  (-\infty,-x_{K}]\cup [x_{K},\infty),\\
 -1/\epsilon_{\circ}<\partial_{x}\kappa  < 1/\epsilon_{\circ},
 \end{array}
\ee
in which $x_{K}\ge K$ is defined in Lemma \ref{lem: construction hat gK}. We omit the dependence of $\kappa$ on $K$ for ease of notations.

\begin{remark}\label{rem: inversion deca kappa} For later use, note that the condition $ |\partial_{x}\kappa |  <  1/\epsilon_{\circ}$  ensures that the map $x\mapsto x+\epsilon \kappa(x)$ and   $x\mapsto x-\epsilon \kappa(x)$ are uniformly strictly increasing for all $0\le \epsilon\le \epsilon_{\circ}$. Also observe that $x_{n}\to x$ and $x'_{n}\in D_{\kappa}^{\epsilon}(x_{n})$, for all $n$, imply that $x'_{n}$ converges to an element $x'\in  D_{\kappa}^{\epsilon}(x)$, after possibly passing to a subsequence.  In particular,  $F_{\kappa}^{\epsilon}$ is continuous.
\end{remark}

When $\kappa\equiv 1$ and $\hat g_{K}^{\epsilon}\equiv \hat g+\epsilon$, \reff{eq: VisEqBdPert}   corresponds to the operator in \reff{eq: VisEqBd} with shaken coefficients, in the traditional terminology of \cite{krylov2000rate}. The function $\kappa$ will be used below to handle the potential linear growth at infinity of $\hat g$. The introduction of the additional approximation  $\hat g_{K}^{\epsilon}$ is motivated by the fact that the proof of Proposition \ref{prop: constr v eps K delta par smoothing} below requires an affine behavior at infinity. As already mentioned, these additional complications do not appear in  \cite{bouchard2013stochastic} because their terminal condition is bounded.

 We now prove that \reff{eq: VisEqBdPert} admits a viscosity solution that remains above the terminal condition $\hat g$ on a small time interval $[T-c^{K}_{\epsilon},T]$. {As already mentioned, we will later  smooth this solution out with a regular kernel, so as  to provide a smooth supersolution of  \reff{eq: VisEqBd}. }

\begin{proposition}\label{prop:existence} For all $\epsilon\in [0,\epsilon_{\circ}]$ and $K>0$, there exists a unique continuous viscosity solution $\vrmghe$ of \reff{eq: VisEqBdPert}
that has linear growth. It
satisfies
\be\label{eq: VisEqreste au dessus hat g + eps pres T}
\vrmghe \ge \hat{g}_{K}+\epsilon/2, \hspace{5mm} \text{ on } [T-c^{K}_{\epsilon}, T] \times\mathbb{R},
\ee
for some  $c^{K}_{\epsilon} \in (0,T)$.

\noindent Moreover,  $\{[\vrmghe]^{+},\epsilon\in [0,\epsilon_{\circ}],K>0\}$ is   bounded by  {a  map with linear growth}, and    $\{[\vrmghe]^{-},\epsilon\in [0,\epsilon_{\circ}],K>0\}$ is  bounded by $\sup g^{-}$.
\end{proposition}

\proof  The proof is mainly a modification of the usual Perron's method, see \cite[Section 4]{CrandallIshiiLions}.

\noindent \textbf{a.} We first prove that there exists two continuous functions   $\bar w$ and $\underline w$ with linear growth that are respectively super- and subsolution of \reff{eq: VisEqBdPert} for any $\epsilon\in [0,\epsilon_{\circ}]$.

Since $\hat g_{K}^{\epsilon}=\hat g_{K}+\epsilon\ge   g$ by Lemma \ref{lem: construction hat gK}, it suffices to set
$$
\underline w:=\inf g>-\infty,$$
see \reff{eq: hyp hat g}.  To construct a supersolution $\bar w$, let us fix   $\eta\in (0,\iota \wedge \inf f^{-1}$) with $\iota$ as in \reff{eq: diff bar gamma f},
set $\tilde \Gamma(x)=\eta x^{2}/2$ and define
$\tilde  g=(\hat g_{K}^{\epsilon_{\circ}}-\tilde  \Gamma)^{\rm conc}+\tilde \Gamma$. Then,  $\tilde g\ge \hat g_{K}^{\epsilon_{\circ}}$, while the same reasoning as in Remark \ref{rem: hat g borne} implies that $\tilde g$ shares the same linear growth as $\hat g_{K}^{\epsilon_{\circ}}$, see \reff{eq: def hat g K eps} and Lemma \ref{lem: construction hat gK}.
We then define $\bar w$ by
$$
\bar w(t,x)=\tilde g(x)+1+(T-t)A
$$
in which
$$
A:= \sup \frac{\sigma^{2}\bar \gamma}{2(1-f\bar \gamma)}  .
$$
The  constant $A$ is finite, and $\bar w$ has the same linear growth as $\tilde g$, see \reff{eq: H1}-\reff{eq: diff bar gamma f}. Since a  concave function is a viscosity supersolution of $-\partial^{2}_{xx}\vp\ge 0$, we deduce that $\tilde g$ is a viscosity supersolution of $\eta-\partial^{2}_{xx}\vp\ge 0$.  Then, $\bar w$ is  a viscosity supersolution of
$$
\min\left\{-\partial_{t}\vp-A\;,\;\eta-\partial^{2}_{xx}\vp\right\}\ge 0.
$$
Since $\bar \gamma\ge \iota\ge \eta$, it remains to use Remark \ref{rem: Feps concave} to conclude that $\bar w$  is a supersolution of  \reff{eq: VisEqBdPert}.

\noindent\textbf{b.}
We now express (\ref{eq: VisEqBdPert}) as a single equation over the whole domain $[0,T]\times\mathbb{R}$ using the following definitions
\begin{align*}
F^{\epsilon,K}_{\kappa,+}(t,x,r,q,M)&:=F_{\kappa}^{\epsilon}(t,x,q,M)\1_{[0,T)}+\max\big\{F_{\kappa}^{\epsilon}(t,x,q,M),r-\hat{g}_{K}^{\epsilon}(x)\big\}\1_{\{T\}} \\
F^{\epsilon,K}_{\kappa,-}(t,x,r,q,M)&:=F_{\kappa}^{\epsilon}(t,x,q,M)\1_{[0,T)}+\min\big\{F_{\kappa}^{\epsilon}(t,x,q,M),r-\hat{g}_{K}^{\epsilon}(x)\big\}\1_{\{T\}}.
\end{align*}
As usual $F^{\epsilon,K}_{\kappa,\pm}[\vp](t,x):=F^{\epsilon,K}_{\kappa,\pm}(t,x,\vp(t,x),\partial_{t}\vp(t,x),\partial^{2}_{xx}\vp(t,x))$. Recall that the  formulations in terms of $F^{\epsilon,K}_{\kappa,\pm}$ lead to the same viscosity solutions as \reff{eq: VisEqBdPert} (see Lemma \ref{lem: pertEquiv} in the Appendix).
 This is the formulation to which we apply Perron's method.
In view of a., the functions   $\underline w$ and $\bar w$  are sub- and supersolution of  $F^{\epsilon,K}_{\kappa,-}=0$ and $F^{\epsilon,K}_{\kappa,+}=0$.
Define:
$$\vrmghe:=\sup\{v \in \text{USC}: \underline w\leq v \leq \bar w\; \text{ and $v$ is a subsolution of } F^{\epsilon,K}_{\kappa,-}=0\},
$$
in which USC denotes the class of upper-semicontinuous maps. Then, the upper- (resp.~lower-)semicontinuous envelope
$(\vrmghe)^{\ast}$ (resp.~$(\vrmghe)_{\ast}$) of $\vrmghe$  is a    viscosity subsolution  of $F^{\epsilon,K}_{\kappa,-}[\vp]=0$  (resp.~supersolution of $F^{\epsilon,K}_{\kappa,+}[\vp]=0$) with linear growth, recall the continuity property of Remark \ref{rem: inversion deca kappa} and see   e.g.~\cite[Section 4]{CrandallIshiiLions}.
The comparison result of Theorem \ref{thm: Comp} stated below implies that
$$
(\vrmghe)^{\ast}=(\vrmghe)_{\ast}, \hspace{5mm} \text{on }[0,T]\times\mathbb{R}.
$$
Hence, $\vrmghe$ is a continuous   viscosity solution of \reff{eq: VisEqBdPert}, recall Lemma \ref{lem: pertEquiv}. By construction, it has linear growth. Uniqueness in this class follows from  Theorem \ref{thm: Comp} again.

\noindent \textbf{c.} It remains to prove \reff{eq: VisEqreste au dessus hat g + eps pres T}. For this, we need a control on the behavior of $\vrmghe$ as $t\to T$. It is enough to obtain it for a lower bound $v_{\epsilon,K}$ that we first construct.   Let $\vp$ be a test function such that
$$\text{(strict)}\min_{[0,T)\times\mathbb{R}}(\vrmghe-\vp)=(\vrmghe-\vp)(t_0, x_0)
$$
for some $(t_{0},x_{0})\in [0,T)\x \R$.
By the supersolution property,
$$\min_{x'\in D^{\epsilon}_{\kappa}(x_{0})}\{\bar \gamma(x')-\partial^{2}_{xx}\vp(t_0, x_0)\} \ge 0.
$$
Recalling \reff{eq: H1} and \reff{eq: diff bar gamma f}, this implies   that, for $x'\in D^{\epsilon}_{\kappa}(x_{0})$,
$$1-f(x')\partial^{2}_{xx}\vp(t_0, x_0)\ge \iota f(x') \ge \iota \inf f =: \tilde{\iota}>0.
$$
Using the supersolution property and the above inequalities yields
\b*
0 &\leq& \min_{x'\in D^{\epsilon}_{\kappa}(x_{0})}\left\{-\partial_{t}\vp(t_0, x_0)-\frac{\sigma^{2}(x')\partial^{2}_{xx}\vp(t_0, x_0)}{2(1-f(x')\partial^{2}_{xx}\vp(t_0, x_0))}\right\} \\
&\leq & \min_{x'\in D^{\epsilon}_{\kappa}(x_{0})}\left\{ -\partial_{t}\vp(t_0, x_0)-\frac{\sigma^{2}(x')\big [\partial^{2}_{xx}\vp(t_0, x_0)-\bar \gamma(x_0)\big ]}{2(1-f(x')\partial^{2}_{xx}\vp(t_0, x_0))} \right\} \\
&\leq& -\partial_{t}\vp(t_0, x_0)-\frac{\tilde{\sigma}^{2} \partial^{2}_{xx}\vp(t_0, x_0)}{2\tilde \iota}+\frac{\tilde{\sigma}^{2} \bar \gamma(x_0)}{2\tilde \iota}
\e*
where $\tilde{\sigma} :=\sup\sigma$.

Denote by $v_{\epsilon,K}$  the unique viscosity solution of
\be\label{eq: Regular}
\left\{-\partial_{t}\vp-\frac{\tilde{\sigma}^{2}\partial^{2}_{xx}\vp}{2\tilde \iota}+\frac{\tilde{\sigma}^{2}\bar \gamma}{2\tilde \iota}\right\}\1_{[0,T)}+(\vp-\hat{g}_{K}^{\epsilon})\1_{\{T\}}=0.
\ee
The comparison principle for (\ref{eq: Regular}) and the Feynman-Kac formula imply that
\b*
\vrmghe(t,x) \ge v_{\epsilon,K}(t,x)=\mathbb{E}\left[-\int_{0}^{T-t}\frac{\tilde{\sigma}^{2} \bar \gamma(S^{x}_r)}{2\tilde \iota}\,dr+\hat{g}_{K}^{\epsilon}(S^{x}_{T-t})\right]
\e*
where  $$\,S^{x}=x+\frac{\tilde{\sigma} }{\sqrt{\tilde \iota}}\,W.
$$

It remains to show that \reff{eq: VisEqreste au dessus hat g + eps pres T} holds for $v_{\epsilon,K}$ in place of $\vrmghe$. The argument is standard. Since  $\hat g_{K}$ is uniformly continuous, see Lemma \ref{lem: construction hat gK}, we can find $B^{K}_{\eps}>0$ such that
$$
\left|\hat{g}_{K}^{\epsilon}(S^{x}_{T-t})-\hat{g}_{K}^{\epsilon}(x)\right|\1_{\{|S^{x}_{T-t}-x|\le B^{K}_{\eps}\}}\le \eps
$$
for all $\eps>0$. We now consider the case $|S^{x}_{T-t}-x|> B^{K}_{\eps}$.   Let $C>0$ denote  a generic constant that does not depend on $(t,x)$ but can change from line to line.    Because, $\hat g_{K}$ is  affine on $[x_{K},\infty)$ and on $(-\infty,-x_{K}]$, see Lemma \ref{lem: construction hat gK},
$$
\Esp{\left|\hat{g}_{K}^{\epsilon}(S^{x}_{T-t})-\hat{g}_{K}^{\epsilon}(x)\right|\1_{\{S^{x}_{T-t}\ge x_{K}  \}}}\le C(T-t)^{\frac12}\mbox{ if } x\ge x_{K},
$$
and
$$
\Esp{\left|\hat{g}_{K}^{\epsilon}(S^{x}_{T-t})-\hat{g}_{K}^{\epsilon}(x)\right|\1_{\{S^{x}_{T-t}\le -x_{K}  \}}}\le C(T-t)^{\frac12}\mbox{ if } x\le -x_{K}.
$$
On the other hand, by linear growth of $\hat{g}_{K}^{\epsilon}$, if $x<x_{K}$, then
\begin{align*}
&\Esp{\left|\hat{g}_{K}^{\epsilon}(S^{x}_{T-t})-\hat{g}_{K}^{\epsilon}(x)\right|\1_{\{S^{x}_{T-t}\ge x_{K}\}}\1_{\{|S^{x}_{T-t}-x|\ge B^{K}_{\eps}\}}}
\\
&\le \Esp{\left|\hat{g}_{K}^{\epsilon}(S^{x}_{T-t})-\hat{g}_{K}^{\epsilon}(x)\right|^{2}}^{\frac12}\Pro{|S^{x}_{T-t}-x|\ge |x_{K}-x|\vee B^{K}_{\eps}}^{\frac12}
\\
&\le C\frac{(1+|x|)(T-t)^{\frac12}}{|x_{K}-x|\vee B^{K}_{\eps}}\le \frac{C}{B^{K}_{\eps}}(T-t)^{\frac12}.
\end{align*}
The  (four) remaining cases are treated similarly, and we obtain
$$
\Esp{\left|\hat{g}_{K}^{\epsilon}(S^{x}_{T-t})-\hat{g}_{K}^{\epsilon}(x)\right|}\le  \frac{C}{B^{K}_{\eps}}(T-t)^{\frac12}+\eps.
$$
 Since $\bar \gamma$ is bounded, this  shows that
$$
|v_{\epsilon,K}(t,x)-\hat g_{K}^{\epsilon}(x)|\le  \frac{C}{B^{K}_{\eps}}(T-t)^{\frac12}+\eps
$$
for $t\in [T-1,T]$. Hence the required result for $v_{\epsilon,K}$.
 Since $\vrmghe\ge v_{\epsilon,K}$, this concludes the proof of \reff{eq: VisEqreste au dessus hat g + eps pres T}.
\ep\\

For later use, note that, by stability, $\vrmghe$ converges to a solution of (\ref{eq: VisEqBd}) when $\epsilon \to 0$ and $K\to \infty$.
\begin{proposition}\label{prop: conv bar v eps to bar v}
As $\epsilon \to 0$ and $K\to \infty$, $\vrmghe$ converges to a function $\vrmgh$ that is the unique   viscosity solution of (\ref{eq: VisEqBd}) with linear growth.
\end{proposition}
\begin{proof}
The family of functions $\{\vrmghe,  \epsilon \in (0,\epsilon_{ \circ}],K>0\}$ is uniformly bounded by  {a map with linear growth}, see Proposition \ref{prop:existence}. In view of the comparison result of Theorem \ref{thm: Comp} below, it suffices to apply  \cite[Theorem 4.1]{Bar94}.
\ep
\end{proof}

\begin{remark}\label{rem: bound en terme w bar underline} The bounds on $\vrmgh$ can be made explicit, which can be useful to design a numerical scheme, see Section \ref{sec: finite diff scheme} below.    First, as a by-product of the proof of Proposition \ref{prop:existence},   $\vrmghe\ge \inf g$. Passing to the limit as $\epsilon\to 0$ and $K\to \infty$ leads to
$$
 \vrmgh \ge \inf g=:\underline w.
$$
We have also obtained that
$$
\vrmghe \le (\hat g_{K}^{\epsilon_{\circ}}-\tilde  \Gamma)^{\rm conc}+\tilde \Gamma +1+A
$$
in which $x\mapsto \tilde \Gamma(x)=\eta x^{2}/2$ for some $\eta \in (0,\iota \wedge \inf f^{-1}$) with $\iota$ as in \reff{eq: diff bar gamma f},   and
$
A:= T\sup (\sigma^{2}\bar \gamma/[2(1-f\bar \gamma)]).
$
On the other hand, \reff{eq: def hat g_{K}} implies
$$
\hat g_{K}^{\epsilon_{\circ}}\le 1+ (2c_{0}+c_{1}|\cdot|-\bar \Gamma^{\circ})^{\rm conc}+\bar \Gamma^{\circ}
$$
for $\bar \Gamma^{\circ}$ such that $\partial^{2}_{xx} \bar \Gamma^{\circ}=\bar \gamma$.
 Then,
\b*
\vrmghe &\le& \left(1+(2c_{0}+c_{1}|\cdot|-\bar \Gamma^{\circ})^{\rm conc}+\bar \Gamma^{\circ}-\tilde  \Gamma\right)^{\rm conc}+\tilde \Gamma +1+A
\\
&\le& \left(1+(2c_{0}+c_{1}|\cdot|-\tilde  \Gamma)^{\rm conc}+\tilde  \Gamma-\tilde  \Gamma\right)^{\rm conc}+\tilde \Gamma +1+A
\\
&=&  \left(1+2c_{0}+c_{1}|\cdot|-\tilde  \Gamma\right)^{\rm conc}+\tilde \Gamma +1+A =:\bar w
\e*
and
$$
 \vrmgh \le \bar w.
$$
The function $\bar w$ defined above can be computed explicitly by arguing as in Remark \ref{rem: hat g borne}.

Also note   that \reff{eq: def hat g_{K}} and the arguments of Remark \ref{rem: hat g borne} imply that there exists a constant $C>0$ such that
\be\label{eq: vrmgh et hat gk same growth}
\limsup_{|x|\to \infty }|\vrmghe(x)|/ (1+|\hat g_{K}(x)|)\le C, \mbox{ for all $\epsilon\in [0,\epsilon_{\circ}]$ and $K>0$.}
\ee
\end{remark}

\subsubsection{Regularization and verification}

Prior to applying our verification argument, it remains to smooth out the function $\vrmghe$.   This is similar to \cite[Section 3]{bouchard2013stochastic}, but here again the fact that $\hat g$ may not be bounded incurs additional difficulties. In particular, we need to use a kernel with a space dependent  window.

 We first fix a smooth kernel
$$
\psi_{\delta}:=\delta^{-2}\psi(\cdot/\delta)
$$
in which  $\delta>0$  and $\psi \in C^{\infty}_{b}$ is a non-negative function with the closure of its support $[-1, 0] \times [-1,1]$   that integrates to $1$, and such that
\be\label{eq: psi sym}
\int y\psi(\cdot,y)dy=0.
\ee
Let us  set
\be\label{eq: def v smoothe}
\vrmgbed(t,x):=\int_{\R\x \R}  \vrmghe([t']^{+},x')\frac{1}{\kappa(x)}\psi_{\delta}\left(t'-t,\frac{x'-x}{\kappa(x)}\right)dt' dx'.
\ee
We recall that $\kappa$ enters into the definition of $F^{\epsilon}_{\kappa}$ and satisfies \reff{eq: cond kappa}.

The following shows that $\vrmgbed$ is a smooth supersolution of \reff{eq: VisEqBd} with a space gradient admitting bounded derivatives. This is due to the space dependent rescaling of the window by $\kappa$ and will be crucial for our verification arguments.

\begin{proposition}\label{prop: constr v eps K delta par smoothing} For all $0<\epsilon<\epsilon_{\circ}$ and $K>0$ large enough, there exists $\delta_{\circ}>0$ such that $\vrmgbed$ is a $C^{\infty}$ supersolution of \reff{eq: VisEqBd} for all $0<\delta<\delta_{\circ}$. It has linear growth and   $\partial_{x} \vrmgbed$ has bounded derivatives of any order.
\end{proposition}

\proof \textbf{a.} It follows from  \reff{eq: cond kappa} and \reff{eq: vrmgh et hat gk same growth} that
\b*
\limsup_{|x|\to \infty }|\vrmghe(x)|/ (1+|\kappa(x)|)<\infty.
\e*
Direct computations and \reff{eq: cond kappa} then show that $\vrmgbed$ has linear growth and that all derivatives of $\partial_{x} \vrmgbed$ are uniformly bounded.
\\
 \textbf{b.} We now prove the supersolution property inside the parabolic domain. Since the proof is very close to that of  \cite[Theorem 3.3]{bouchard2013stochastic}, we only provide the   arguments that require to be adapted, and refer to their proof for other elementary details. Fix $\ell>0$ and  set
$$
v_{\ell}(t,x):=\vrmgbed(t,(-\ell)\vee x\wedge \ell).
$$
We omit the superscripts that are superfluous in this proof.
Given $k\ge 1$, set
$$
v_{\ell,k}(z):=\inf_{z'\in  [0,T]\x\R} \left(v_{\ell}(z')+k|z-z'|^{2}\right).
$$
 Since $v_{\ell}$ is bounded and continuous, the infimum in the above is achieved by a point  $\hat z_{\ell,k}(z)=(\hat t_{\ell,k}(z),\hat x_{\ell,k}(z))$, and   $v_{\ell,k}$ is bounded, uniformly in $k\ge 1$. This implies that  we can find $C_{\ell}>0$, independent of $k$, such that
\be\label{eq: def rho k}
|z-\hat z_{\ell,k}(z)|^{2}\le C_{\ell}/k=:(\rho_{\ell,k})^{2}.
\ee
Moreover, a simple change of variables argument shows that, if $\vp$ is a smooth function such that $v_{\ell,k}-\vp$ achieves a minimum at $z\in [0,T)\x (-\ell,\ell)$, then
$$
(\partial_{t}\vp,\partial_{x}\vp,\partial^{2}_{xx}\vp)(z)\in \bar \Pc^{-} v_{\ell}(\hat z_{\ell,k}(z)),
$$
where  $\bar \Pc^{-} v_{\ell}(\hat z_{\ell,k}(z))$ denotes the \emph{closed} parabolic subjet of $v_{\ell}$ at $\hat z_{\ell,k}(z)$; see e.g.~\cite{CrandallIshiiLions} for the definition. Then, Proposition \ref{prop:existence} implies that $v_{\ell,k}$ is a supersolution of
\b*
\min_{x'\in D^{\epsilon}_{\kappa}(\hat x_{\ell,k} )} \min\left\{-\partial_{t}\vp-\frac{\sigma^{2}(x')\partial^{2}_{xx}\vp}{2(1-f(x')\partial^{2}_{xx}\vp)}, \bar \gamma(x')-\partial^{2}_{xx}\vp\right\} \ge 0
\e*
on $[\rho_{\ell,k},T-\rho_{\ell,k})\x (-\ell+\rho_{\ell,k},\ell-\rho_{\ell,k})$. We next deduce from \reff{eq: def rho k} that  $x'\in D^{\epsilon/2}_{\kappa}(x)$  implies
\b*
-\frac{\epsilon}{2} \kappa(x')-C_{\ell}/k^{\frac12}\le \hat x_{\ell,k}(t,x)-x'\le \frac{\epsilon}{2} \kappa(x')+C_{\ell}/k^{\frac12}.
\e*
 Since $\inf \kappa>0$, this shows that $x'\in D^{\epsilon}_{\kappa}(\hat x_{\ell,k}(t,x) )$ for $k$ large enough with respect to $\ell$. Hence, $v_{\ell,k}$ is a supersolution of
\b*
\min_{x'\in D^{\epsilon/2}_{\kappa}} \min\left\{-\partial_{t}\vp-\frac{\sigma^{2}(x')\partial^{2}_{xx}\vp}{2(1-f(x')\partial^{2}_{xx}\vp)}, \bar \gamma(x')-\partial^{2}_{xx}\vp\right\} \ge 0
\e*
on $[\rho_{\ell,k},T-\rho_{\ell,k})\x (-\ell+\rho_{\ell,k},\ell-\rho_{\ell,k})$.

We now argue as in \cite{ishii1995equivalence}. Since $v_{\ell,k}$ is semi-concave, there exist $\partial_{xx}^{2,abs}v_{\ell,k} \in L^{1}$ and  a Lebesgue-singular negative Radon measure $ \partial_{xx}^{2,sing} v_{\ell,k}$ such that
$$
\partial^{2}_{xx}v_{\ell,k}(dz)=\partial_{xx}^{2,abs}v_{\ell,k}(z)dz+\partial_{xx}^{2,sing} v_{\ell,k}(dz)\mbox{ in the distribution sense}
$$
and
$$
(\partial_{t} v_{\ell,k},\partial_{x}v_{\ell,k},\partial_{xx}^{2,abs}v_{\ell,k})\in \bar \Pc^{-}v_{\ell,k} \mbox{ a.e. on } [\rho_{k},T-\rho_{k}]\x (-\ell+\rho_{\ell,k},\ell-\rho_{\ell,k}),$$
 see \cite[Section 3]{jensen1988maximum}.
Hence, the above implies that
\begin{equation*}
\min_{x'\in D^{\epsilon/2}_{\kappa} } \min\left\{-\partial_{t}v_{\ell,k}-\frac{\sigma^{2}(x')\partial_{xx}^{2,abs}v_{\ell,k}}{2(1-f(x')\partial_{xx}^{2,abs}v_{\ell,k})}, \bar \gamma(x')-\partial_{xx}^{2,abs}v_{\ell,k}\right\} \ge 0
\end{equation*}
a.e.~on  $[\rho_{\ell,k},T-\rho_{\ell,k})\x (-\ell+\rho_{\ell,k},\ell-\rho_{\ell,k})$, or equivalently, by \reff{eq: def D eps},
\begin{equation*}
 \min\left\{-\partial_{t}v_{\ell,k} -\frac{\sigma^{2}(x)\partial_{xx}^{2,abs}v_{\ell,k}}{2(1-f(x)\partial_{xx}^{2,abs}v_{\ell,k})}, \bar \gamma(x ) -\partial_{xx}^{2,abs}v_{\ell,k} \right\}(t',x') \ge 0
\end{equation*}
for all $x$  and for a.e.~$(t',x')\in [\rho_{\ell,k},T-\rho_{\ell,k})\x(-\ell+\rho_{\ell,k},\ell-\rho_{\ell,k})$ such that $ 2|x'-x|\le \epsilon\kappa(x)$.
Take $0<\delta<\eps/2$. Integrating the previous inequality with respect to $(t',x')$ with the kernel function  $\psi_{\delta}(\cdot,\cdot/\kappa)/\kappa$, using the concavity and monotonicity property of Remark \ref{rem: Feps concave} and the fact that $\partial_{xx}^{2,sing} v_{\ell,k}$ is non-positive, we obtain
\begin{equation}\label{eq: pde bar v k ae}
 \min\left\{-\partial_{t}v^{\delta}_{\ell,k}-\frac{\sigma^{2} \partial_{xx}^{2}v^{\delta}_{\ell,k}}{2(1-f \partial_{xx}^{2}v^{\delta}_{\ell,k})}, \bar \gamma -\partial_{xx}^{2}v^{\delta}_{\ell,k}\right\} \ge 0
\end{equation}
on  $[\rho_{\ell,k}+\delta,T-\rho_{\ell,k})\x(-  x^{-}_{\ell,k},  x^{+}_{\ell,k})$, in which
$$
v^{\delta}_{\ell,k}(t,x):=\int_{\R\x \R}  v_{\ell,k}([t']^{+},x')\frac{1}{\kappa(x)}\psi_{\delta}\left(t'-\cdot,\frac{x'-\cdot}{\kappa(x)}\right)dt' dx'
$$
and
$$
  x^{+}_{\ell,k}+\frac{\delta}{2} \kappa(  x^{+}_{\ell,k})= \ell-\rho_{\ell,k}\mbox{ and }   -x^{-}_{\ell,k}-\frac{\delta}{2} \kappa(-  x^{-}_{\ell,k})=- \ell+\rho_{\ell,k}.
$$
The above are well defined, see Remark \ref{rem: inversion deca kappa}.
By Remark \ref{rem: inversion deca kappa} and \reff{eq: def rho k}, $   \pm x^{\pm}_{\ell,k}\to \pm\infty$ and $\rho_{\ell,k}\to 0$ as $k\to \infty$ and then $\ell \to \infty$. Moreover, $v^{\delta}_{\ell,k}\to \vrmgbed$ as $k\to \infty$ and then $\ell \to \infty$, and the derivatives also converge. Hence, \reff{eq: pde bar v k ae} implies  that $\vrmgbed$ is a supersolution of \reff{eq: VisEqBd} on $[\delta,T)\x \R$.
\\
 \textbf{c.} We conclude by discussing the boundary condition at $T$.   We know from Proposition \ref{prop:existence} that
 \b*
\vrmghe \ge \hat{g}_{K}+\epsilon/2, \hspace{5mm} \text{ on } [T-c^{K}_{\epsilon}, T] \times\mathbb{R}.
\e*
Since $\hat g$ is uniformly continuous, see \reff{eq: hyp hat g}, so is $\hat{g}_{K}$, and therefore $\vrmgbed(T,\cdot)\ge \hat{g}_{K}$ on the compact set $[-2x_{K} ,2 x_{K}]$ for $\delta>0$ small enough with respect to $\epsilon$, see Lemma \ref{lem: construction hat gK} for the definition of $x_{K}\ge K$. Now observe that $x\ge 2 x_{K}$ and $|x'-x|\le \delta \kappa(x)$  imply that $x'\ge 2 x_{K}(1-\delta c^{K}_{1})-\delta c^{K}_{0}$ in which $c^{K}_{1}$ and $c^{K}_{0}$ are constants. This actually follows from the affine behavior of $\kappa$ on   $[x_{K},\infty)$, see  \reff{eq: cond kappa} and Lemma  \ref{lem: construction hat gK}. For $\delta$ small enough, we then obtain $x'\ge  x_{K}$. Since  $\hat{g}_{K}$ is affine on $[x_{K},\infty)$, and since $\psi$ is symmetric in its second argument, see \reff{eq: psi sym}, it follows that
 $$
\vrmgbed(T,x)\ge \int_{\R\x \R}  \hat{g}_{K}(x')\frac{1}{\kappa(x)}\psi_{\delta}\left(t'-T,\frac{x'-x}{\kappa(x)}\right)dt' dx'=\hat{g}_{K}(x)
 $$
 for all $x\ge 2x_{K}$. This also holds for $x\le -2x_{K}$, by the same arguments.
\ep\\

We can now use a verification argument and provide the main result of this section.

\begin{theorem}\label{thm:Reg} Let $\vrmgh$  be defined as in Proposition \ref{prop: conv bar v eps to bar v}. It has linear growth. Moreover,  $\vrmgh\ge \vrmg$ on $[0,T]\times \mathbb{R}$.
\end{theorem}
\begin{proof} The linear growth property has already been stated in Proposition \ref{prop: conv bar v eps to bar v}.
We now show that    $\vrmgh\ge \vrmg$ by  applying a verification argument to $\vrmgbed$. From now on $0<\epsilon\le \epsilon_{\circ}$ in which $ \epsilon_{\circ}$ is as in \reff{eq: cond kappa}. The parameters $K,\delta >0$ are chosen as in Proposition \ref{prop: constr v eps K delta par smoothing}.

  Fix $(t,x)\in (0,T)\x \R$ and $\delta \in (0,t\wedge \epsilon)$.  Let $(X,Y,V)$ be defined as in \reff{eq: S lim conti}-\reff{eq: def Y sans saut}-\reff{eq: V lim conti} with  $(x,\partial_{x}\vrmgbed(t,x),\vrmgbed(t,x) -\partial_{x}\vrmgbed(t,x) x)$ as initial condition at $t$, and for the Markovian controls
\begin{align*}
\hat a&=\left(\frac{\sigma \partial^{2}_{xx}\vrmgbed }{1-f \partial^{2}_{xx}\vrmgbed }\right)(\cdot,X)\\
\hat b&=\left(\frac{\partial^{2}_{tx}\vrmgbed+\partial^{2}_{xx}\vrmgbed (\mu+\hat a\sigma f')+\frac12\partial^{3}_{xxx}\vrmgbed(\sigma+\hat af)^{2}}{1-f\partial^{2}_{xx}\vrmgbed}\right)(\cdot,X).
\end{align*}
By definition of $F$, \reff{eq: diff bar gamma f} and \reff{eq: H1}, the above is well-defined as the denominators are always bigger than $\inf f \iota>0$. All the involved functions being bounded  and Lipschitz, see Proposition \ref{prop: constr v eps K delta par smoothing},  it is easy to check that a solution to the corresponding stochastic differential equation exists, and  that   $(\hat a,\hat b)\in \Ac^{\circ}$. Direct computations then show that $Y=\partial_{x}\vrmgbed(\cdot,X)$. Moreover, the fact that $\vrmgbed$ is a  supersolution of $F[\vp]=0$ on $[t,T]\x \R$ ensures that {the gamma constraint \reff{eq: contrainte gamma} holds}, for some $k\ge 1$,  and that
 \b*
 -\partial_{t}\vrmgbed(\cdot,X) - \frac12   \sigma(X) \hat a\ge 0
\;\mbox{ on } [t,T).
 \e*
The last inequality combined with the definition of $\hat a$ implies
\begin{align*}
\frac{1}{2}f(X)\hat a^{2}&\ge
\partial_{t}\vrmgbed(\cdot,X) + \frac12 (\sigma(X)+f(X)\hat a)\hat a\\
&=\partial_{t}\vrmgbed(\cdot,X) + \frac12 (\sigma_{X}^{\hat a}(X))^{2} \partial^{2}_{xx}\vrmgbed(\cdot,X) \;\mbox{ on } [t,T).
\end{align*}

Hence,
\b*
V_T &=& \vrmgbed(t,x)+\frac 12 \int_{t}^{T}f(X_u)\hat a_{u}^{2}\,du +\int_{t}^{T}\partial_{x}\vrmgbed(u,X_u)\,dX_u \\
&\ge & \vrmgbed(t,x)+\ \int_{t}^{T}d \vrmgbed(u,X_u) \\
&=&\vrmgbed(T,X_T) \ge g(X_T),
\e*
in which the last inequality follows from Proposition \ref{prop: constr v eps K delta par smoothing}  again.

 It remains to pass to the limit $\delta,\epsilon\to 0$. By Proposition \ref{prop:existence}, $ \vrmghe$ is continuous, so that $\vrmgbed$ converges pointwise to $\vrmghe$ as $\delta\to 0$. By  Proposition \ref{prop: conv bar v eps to bar v},  $\vrmghe$ converges pointwise to $\vrmgh$  as $\epsilon \to 0$ and $K \to \infty$. In view of the above this implies the required result: $\vrmgh\ge \vrmg$.
\ep
\end{proof}

\begin{remark}\label{rem: borne sup vk par verification} Note that, in the above proof, we have constructed a super-hedging strategy in $\Ac_{k,\bar \gamma}(t,x)$ and starting with $|Y_{t}|\le k$, for some $k\ge 1$ which can be chosen in a uniform way with respect to $(t,x)$, while $\vrmgbed$ has linear growth.
\end{remark}

\subsubsection{Comparison principle}\label{sec: comp}

We provide here  the comparison principle that was used several times in the above.
Before stating it, let us make the following observation, based on direct computations. Recall \reff{eq: H1} and \reff{eq: diff bar gamma f}.

\begin{proposition}\label{prop: regu terme fully non linear} Fix $\rho>0$. Consider the map
$$
(t,x,M)\in[0,T]\x \R\x \R \mapsto\Psi(t,x,M) =\frac{\sigma^{2}(x)M}{2(e^{\rho t}-f(x)M)}.
$$
Then, $M\mapsto  \Psi(t,x,M)$ is continuous, uniformly in $(t,x)$, on
$$
O:=\{(t,x,M)\in [0,T]\x \R\x \R :  M\le e^{\rho t}\bar  \gamma(x)\}.
$$
Moreover, there exists $L>0$ such that  $x\mapsto  \Psi(t,x,M)$ is $L$-Lipschitz on  $O$.
\end{proposition}

\begin{theorem}\label{thm: Comp} Fix  $\epsilon\in  [0,\epsilon_{\circ}]$.
Let $U$ (resp.~$V$) be a upper semicontinuous viscosity subsolution (resp.~lower semicontinuous supersolution) of $F_{\kappa}^{\epsilon}=0$ on $[0,T)\x \R$. Assume that $U$ and $V$ have linear growth and that  $U\le V$ on $\{T\}\x \R$, then   $U\leq V$ on $[0,T] \times \mathbb{R}$.
\end{theorem}

\begin{proof}\noindent    Set $\hat U(t,x):=e^{\rho t}U(t,x)$, $\hat V(t,x):=e^{\rho t}V(t,x)$. Then, $\hat U$ and $\hat V$  are respectively sub- and supersolution of
\be\label{eq: replacedComp}
\min_{x'\in D^{\epsilon}_{\kappa}}\min\left\{\rho\vp-\partial_{t}\vp-\frac{\sigma^{2}(x')\partial_{xx}\vp}{2(e^{\rho t}-f(x')\partial_{xx}\vp)}, e^{\rho t}\bar \gamma(x')-\partial_{xx}\vp\right\}=0
\ee
on $ [0,T) \times \mathbb{R}.$ For later use, note that the infimum over $D^{\epsilon}_{\kappa}$ is achieved in the above, by the continuity of the involved functions.

If  $\sup_{[0,T]\times\mathbb{R}}(\hat U-\hat V)>0$, then we can find $\lambda \in (0,1)$ such that $\sup_{[0,T]\times\mathbb{R}}(\hat U-\hat V_{\lambda}) >0$ with $\hat V_{\lambda}:=\lambda \hat V+(1-\lambda) w$, in which
$$
w(t,x):= (T-t)A+ (c^{U}_{0}+c_{1}^{U}|\cdot|-\frac{\iota}{4}|\cdot|^{2})^{\rm conc}(x)+\frac{\iota}{4}|x|^{2}
$$
with $c^{U}_{0},c_{1}^{U}$ two constants such that
$
e^{\rho T}| U|\le c^{U}_{0}+c_{1}^{U}|\cdot| $ and
$$
A:=\frac12 \sup\frac{\sigma^{2}}{1-\frac{\iota}{2} f}\frac{\iota}{2},
$$
where  $\iota>0$ is as in \reff{eq: diff bar gamma f}.
Note that
\be\label{eq: hat v lamb T ge hat u T}
\hat V_{\lambda}(T,\cdot)\ge \hat U(T,\cdot),
\ee
and that
\be\label{eq: prop sur sol conv Vlambda}
\begin{array}{c}w \mbox{ is a viscosity supersolution of \reff{eq: replacedComp}}
\\
\mbox{$\hat V_{\lambda}$ is a viscosity supersolution of } \lambda \bar \gamma +(1-\lambda)\frac{\iota}{2}-\partial_{xx}^{2}\vp\ge 0.
\end{array}
\ee
Moreover,
 by Remark \ref{rem: Feps concave},  $\hat V_{\lambda}$ is a  supersolution of \reff{eq: replacedComp}.
Define for $\eps>0$ and $n\ge 1$
\begin{equation}\label{eq: Thetan}
\Theta^{\eps}_n := \sup_{(t,x,y)\in [0,T]\times\mathbb{R}^{2}}\Big[\hat U(t,x)-\hat V_{\lambda}(t,y)-\left(\frac{\eps}{2} |x|^{2}+\frac{n}{2} |x-y|^{2}\right)\Big]=: \eta>0,
\end{equation}
in which the last inequality holds for $n>0$ large enough and $\eps>0$ small enough.
Denote by  $(t^{\eps}_n,x^{\eps}_n,y^{\eps}_n)$ the point at which this supremum is achieved. By \reff{eq: hat v lamb T ge hat u T}, it must hold that $t^{\eps}_{n}<T$, and, by standard arguments, see e.g.,~\cite[Proposition 3.7]{CrandallIshiiLions},
\be\label{eq: CompLim}
\lim\limits_{n \to \infty} n|x^{\eps}_{n}-y^{\eps}_{n}|^{2}=0.
\ee
Moreover,  Ishii's lemma implies the existence of $(a^{\eps}_{n},M^{\eps}_n, N^{\eps}_n)\in \R^{3}$ such that
\b*
&\left(a^{\eps}_{n},\eps x^{\eps}+n(x^{\eps}_n-y^{\eps}_n), M^{\eps}_n\right)\in \bar {\cal P}^{2,+}\hat U(t^{\eps}_n,x^{\eps}_n)\\
&\left(a^{\eps}_{n},-n(x^{\eps}_n-y^{\eps}_n), N^{\eps}_n\right)\in \bar {\cal P}^{2,-}\hat V_{\lambda}(t^{\eps}_n,y^{\eps}_n),
\e*
in which $\bar{\cal P}^{2,+}$ and  $\bar{\cal P}^{2,-}$ denote as usual the {\sl closed} parabolic super- and subjets, see \cite{CrandallIshiiLions}, and
\b*
\begin{pmatrix}
M^{\eps}_n & 0 \\
0 & -N^{\eps}_n
\end{pmatrix}
\leq
R^{\eps}_{n}+\frac1n (R^{\eps}_{n})^{2}=
3n
\begin{pmatrix}
1 & -1 \\
-1 & 1
\end{pmatrix}
+
\begin{pmatrix}
3\eps+\frac{{\eps^{2}}}{n} & -\eps \\
-\eps & 0
\end{pmatrix}
\e*
with
\b*
R^{\eps}_{n}:= n
\begin{pmatrix}
1 + \frac{\eps}{n}& -1 \\
-1 & 1
\end{pmatrix}.
\e*
In particular,
\be\label{eq: diff Xn Yn}
M^{\eps}_{n}-N^{\eps}_{n}\le \delta^{\eps}_{n}\mbox{ with } \delta^{\eps}_{n}:=\eps+\frac{\eps^{2}}{n}.
\ee
Then, by \reff{eq: prop sur sol conv Vlambda} and  \reff{eq: diff bar gamma f},
\begin{align}\label{eq: contrainte gamma Yn Xn}
0<(1-\lambda)\frac{\iota}{2}\le   e^{\rho t^{\eps}_{n}} \bar \gamma(\hat y^{\eps }_{n} )-N^{\eps}_n\le  e^{\rho t^{\eps}_{n}} \bar \gamma(\hat y^{\eps}_{n} )-M^{\eps}_n+\delta^{\eps}_{n},
\end{align}
in which $\hat y^{\eps }_{n} \in D^{\epsilon}_{\kappa}(y^{\eps}_{n} )$. In view of Remark \ref{rem: inversion deca kappa}, this shows that $e^{\rho t^{\eps}_{n}} \bar \gamma(\hat x^{\eps }_{n} )-M^{\eps}_n>0$ for some $\hat x^{\eps}_{n} \in D^{\epsilon}_{\kappa}(x^{\eps}_{n} )$, for $n$ large enough and $\eps$ small enough, recall \reff{eq: CompLim}. Hence, the super- and subsolution properties of $\hat V_{\lambda}$ and $\hat U$ imply that we can find $u^{\eps}_{n}\in [-\epsilon,\epsilon]$ together with $\hat  y^{\eps }_{n}$ and $\hat x^{\eps }_{n}$ such that
\be\label{eq: comp def hat y eps n et hat x eps n}
\hat  y^{\eps }_{n}+u^{\eps}_{n}\kappa(\hat  y^{\eps }_{n})= y^{\eps }_{n}\;,\;\hat  x^{\eps }_{n}+u^{\eps}_{n}\kappa(\hat  x^{\eps }_{n})= x^{\eps }_{n}
\ee
and
\b*
\rho (\hat U(t^{\eps}_{n},x^{\eps}_{n})-\hat V_{\lambda}(t^{\eps}_{n},y^{\eps}_{n}))\leq \frac{\sigma^{2}(\hat  x^{\eps }_{n})M^{\eps}_n}{2(e^{\rho t^{\eps}_{n}}-f(\hat  x^{\eps }_{n})M^{\eps}_n)}-\frac{\sigma^{2}(\hat  y^{\eps }_{n})N^{\eps}_n}{2(e^{\rho t^{\eps}_{n}}-f(\hat  y^{\eps }_{n})N^{\eps}_n)}.
\e*
By Remark \ref{rem: Feps concave} and \reff{eq: diff Xn Yn}, this shows that
\begin{align*}
&\rho (\hat U(t^{\eps}_{n},x^{\eps}_{n})-\hat V_{\lambda}(t^{\eps}_{n},y^{\eps}_{n}))\\
&\leq \frac{\sigma^{2}(\hat  x^{\eps }_{n})(N^{\eps}_n+\delta^{\eps}_{n})}{2(e^{\rho t^{\eps}_{n}}-f(\hat  x^{\eps }_{n})(N^{\eps}_n+\delta^{\eps}_{n}))}-\frac{\sigma^{2}(\hat  y^{\eps }_{n})N^{\eps}_n}{2(e^{\rho t^{\eps}_{n}}-f(\hat  y^{\eps }_{n})N^{\eps}_n)}.
\end{align*}
It remains to apply Proposition \ref{prop: regu terme fully non linear} together with \reff{eq: contrainte gamma Yn Xn} for $n$ large enough and $\eps$ small enough to obtain
\begin{align*}
&\rho (\hat U(t^{\eps}_{n},x^{\eps}_{n})-\hat V_{\lambda}(t^{\eps}_{n},y^{\eps}_{n}))\\
&\leq    \frac{\sigma^{2}(\hat  x^{\eps }_{n})N^{\eps}_n}{2(e^{\rho t^{\eps}_{n}}-f(\hat  x^{\eps }_{n})N^{\eps}_n)}-\frac{\sigma^{2}(\hat  y^{\eps }_{n})N^{\eps}_n}{2(e^{\rho t^{\eps}_{n}}-f(\hat  y^{\eps }_{n})N^{\eps}_n)}+{O^{\eps}_{n}(1)}\\
&\leq L\left|\hat x^{\eps}_{n}-\hat y^{\eps}_{n} \right|  +{O^{\eps}_{n}(1)}
\end{align*}
for some $L>0$ {and where $O^{\eps}_{n}(1)\to 0$ as $n\to \infty$ and then $\eps\to 0$}. By continuity and \reff{eq: CompLim} combined with Remark \ref{rem: inversion deca kappa} and \reff{eq: comp def hat y eps n et hat x eps n}, this contradicts  \reff{eq: Thetan} for $n$ large enough.\ep
\end{proof}

\subsection{Supersolution property for the weak formulation}\label{sec: weak form}

In this part, we provide a lower bound $\vrmgl$ for  $\vrmg$ that is a supersolution of \reff{eq: VisEqBd}. It is constructed by considering a weak formulation of the stochastic target problem \reff{eq: def vgamma}  in the spirit of  \cite[Section 5]{cheridito2005multi}. Since our methodology is slightly different, we provide the main arguments.

  On $C(\R_{+})^{5}$, let us now denote by $(\tilde \zeta:=(\tilde a,\tilde b,\tilde \alpha,\tilde \beta), \tilde W)$ the coordinate process and let $\tilde {\mathbb F}^{\circ}=(\tilde \Fc^{\circ}_{s})_{  s \le T}$ be its raw filtration.  We say that a probability measure $\tilde \P$ belongs to $\tilde \Ac_{k}$ if $\tilde W$ is a $\tilde \P$-Brownian motion and  if for all $0\le \delta\le 1$ and $r\ge 0$ it holds $\tilde \P$-a.s.~that
\be\label{eq: cond weak 1}
\tilde a=\tilde a_{0}+\int_{0}^{\cdot}\tilde \beta_{s} ds +\int_{0}^{\cdot} \tilde \alpha_{s} d\tilde W_{s}\;\mbox{ for some } \tilde a_{0}\in \R,
\ee
\be\label{eq: cond weak 2}
 \sup_{\R_{+}} |\tilde \zeta|\le k\;,\;
\ee
and
\be\label{eq: cond weak 3}
\E^{{\tilde \P}}\left[\sup\left\{|\tilde \zeta_{s'}-\tilde \zeta_{s}|,\; r\le s \le s'\le s+\delta\right\}|{\tilde \Fc_{r}^{\circ}} \right]\le k\delta.
\ee
For $\tilde \phi:=(y,\tilde a,\tilde b)$, $y\in \R$, we define $(\tilde X^{x,\tilde \phi},\tilde Y^{\tilde \phi}, \tilde V^{x,v,\tilde \phi})$ as in \reff{eq: S lim conti}-\reff{eq: def Y sans saut}-\reff{eq: V lim conti} associated to the control $(\tilde a,\tilde b)$ with time-$0$ initial condition $(x,y,v)$, and  with $\tilde W$ in place of $W$. {For $t\le T$ and $k\ge 1$, we} say that $\tilde \P\in \tilde \Gc_{k,\bar \gamma}(t,x,v,y)$ if
\begin{equation} \label{eq: def sur rep sous contrainte faible}
\left[\tilde  V^{x,v,\tilde \phi}_{T-t}\ge g(\tilde X^{x,\tilde \phi}_{T-t}) \;\mbox{ and } \;  -k\le \gamma^{\tilde a}_{Y}(\tilde X^{x,\tilde \phi})\le \bar \gamma (\tilde X^{x,\tilde \phi})\mbox{ on } \R_{+}\right]\;\;\tilde \P-{\rm a.s.}
 \end{equation}
We finally define
\begin{equation*}\label{eq: def v weak k}
   \vrmglk(t,x):=\inf\{ v=c+ yx~:~(c,y)\in \R\times [-k,k] \mbox{ s.t. }   \tilde \Ac_{k}\cap   \tilde \Gc_{k,\bar \gamma}(t,x,v,y)\ne \emptyset\},
 \end{equation*}
 and
\begin{equation}\label{eq: def v low}
   \vrmgl(t,x):=\liminf_{\tiny \begin{array}{c}(k,t',x')\to (\infty,t,x)\\(t',x')\in [0,T)\x \R\end{array}} \vrmglk(t',x'),\;\;\;(t,x)\in [0,T]\x \R.
 \end{equation}

The following is an immediate consequence of our definitions.
\begin{proposition}\label{prop: v weak is lower bound}
$\vrmg\ge \vrmgl$  on $[0,T)\x \R$.
\end{proposition}

In the rest of this section, we show that $\vrmgl$ is a viscosity supersolution of \reff{eq: VisEqBd}.
We start with an easy remark.

\begin{remark}\label{rem: borne a et sigma a}
Observe that the gamma constraint in \reff{eq: def sur rep sous contrainte faible} implies that we can find $\eps>0$ such that
$$
 \frac{\eps}{1+k\eps^{-1}}\le \sigma^{\tilde a}_{X}(\tilde X^{x,\tilde \phi})\le \eps^{-1}+\eps^{-2} \;\mbox{ and } |\tilde a| \le  \eps^{-1}\;\;\tilde \P-{\rm a.s.},
$$    for all $\tilde \P\in \tilde \Ac_{k}\cap   \tilde \Gc_{k,\bar \gamma}(t,x,v,y)$ and $k\ge 1$. Indeed, if $\tilde a\ge  -\sigma/f$ then
$-k\le \gamma^{\tilde a}_{Y}\le \bar \gamma$  implies
\begin{align*}
&  (-\frac{k\sigma}{1+kf})\vee (-\frac{\sigma}{f}) \le\tilde  a \le \frac{\bar \gamma\sigma}{1-\bar \gamma f} \; \mbox{ and } \;\tilde af+\sigma \ge \sigma/(1+kf).
\end{align*}
Then our claim follows from  \reff{eq: H1}-\reff{eq: diff bar gamma f}. On the other hand, if $\sigma +\tilde af< 0$, then   $\gamma^{\tilde a}_{Y} \le \bar \gamma$ implies $\tilde a\ge \bar \gamma\sigma/(1-f\bar \gamma)\ge 0$, see \reff{eq: diff bar gamma f}, while $\tilde a<-f/\sigma<0$, a contradiction.
\end{remark}

We then  show that $\vrmglk$ has linear growth, for $k$ large enough.

\begin{proposition}\label{prop: w low k linear growth} There exists $k_{o}\ge 1$ such that $\{|\vrmglk|,k\ge k_{o}\}$ is uniformly bounded from above by a continuous map with linear growth.
\end{proposition}

\proof \textbf{a.}  First note that   Remark \ref{rem: borne sup vk par verification} implies that  $\{(\vrmglk)^{+},$ $k\ge k_{o}\}$ is uniformly bounded from above by a map with linear growth, for some $k_{o}$ large enough.
\\
\textbf{b.}  Let us now fix $\tilde \P\in \tilde \Ac_{k}\cap   \tilde \Gc_{k,\bar \gamma}(t,x,v,y)$. Using Remark \ref{rem: borne a et sigma a} combined with  \reff{eq: H1} and the condition that $(\tilde a,\tilde b,\tilde \alpha,\tilde \beta)$ is $\tilde \P$-essentially bounded, one can find  $\check \P\sim \tilde \P$ under which  $\int_{0}^{\cdot} \tilde Y^{\tilde \phi}_{s}d\tilde X^{x,\tilde \phi}_{s}$ is a martingale on $[0,T-t]$. Then,  the condition $\tilde V^{x,v,\tilde \phi}_{T-t}\ge g(\tilde X^{x,\tilde \phi}_{T-t})$ $\tilde \P$-a.s.~implies
$v+\E^{\check \P}[\frac12\int_{0}^{T-t}\tilde a^{2}_{s}f(\tilde X^{x,\tilde \phi}_{s})ds]\ge \inf g>-\infty$, recall \reff{eq: hyp hat g}. By Remark \ref{rem: borne a et sigma a} and \reff{eq: H1},
$v\ge \inf g- C>-\infty$, for some constant $C$ independent of $\tilde \P\in  \cup_{k} (\tilde \Ac_{k}\cap   \tilde \Gc_{k,\bar \gamma}(t,x,v,y))$. Hence $\{(\vrmglk)^{-},$ $k\ge k_{o}\}$ is bounded  by a constant.
\ep
\\

We now prove that existence holds in the problem defining $\vrmglk$ and that it is lower-semicontinuous.

\begin{proposition}\label{prop: existence faible} For all $(t,x)\in [0,T]\x \R$ and $k\ge 1$ large enough,  there exists $(c,y)\in \R\times [-k,k] $ such that
$ \vrmglk(t,x)=c+ yx$ and $ \tilde \Ac_{k}\cap \tilde \Gc_{k,\bar \gamma}(t,c+xy,y)\ne \emptyset$. Moreover,   $\vrmglk$ is lower-semicontinuous for each $k\ge 1$ large enough.
\end{proposition}

\proof    By \cite[Proposition XIII.1.5]{revuz1999continuous} and the condition \reff{eq: cond weak 3} taken for $r=0$, the set $\tilde \Ac_{k}$ is weakly relatively compact.  Moreover,   \cite[Theorem 7.10 and Theorem 8.1]{kurtz1996weak} implies that any limit point $( \P_{*},t_{*},x_{*},c_{*},y_{*})$ of a sequence $(\P_{n},t_{n},x_{n},c_{n},y_{n})_{n\ge 1}$ such that $\P_{n}\in \tilde \Ac_{k}\cap   \tilde \Gc_{k,\bar \gamma}(t_{n},x_{n},c_{n}+x_{n}y_{n},y_{n}) $ for each $n\ge 1$  satisfies    $\P_{*}\in\tilde  \Ac_{k}\cap   \tilde \Gc_{k,\bar \gamma}(t_{*},x_{*},c_{*}+x_*y_{*},y_{*})$.
{Since $\vrmglk$ is locally bounded, by Proposition \ref{prop: w low k linear growth} when $k\ge k_{o}$,} the announced existence and lower-semicontinuity readily follow.
\ep
\\

We can finally prove the main result of this section.

\begin{theorem}\label{thm: weak v is supersolution}  The function $\vrmgl$ is a viscosity supersolution of  \reff{eq: VisEqBd}. It has linear growth.
\end{theorem}

\proof  The linear growth property is an immediate consequence of  the uniform linear growth of $\{|\vrmglk|,k\ge k_{o}\}$ stated in Proposition \ref{prop: w low k linear growth}.  To prove the supersolution property, it  suffices to show that it holds for each $\vrmglk$, with $k\ge k_{o}$,  and then to apply standard stability results, see e.g.~\cite{Bar94}.

\noindent{\bf a.} We first prove the supersolution property on $[0,T)\x \R$. We adapt the arguments of \cite{cheridito2005multi} to our context.
 Let us   consider a $C^{\infty}_{b}$ test function $\vp$ and  $(t_{0},x_{0})\in [0,T)\x \R$ such that
$$
\text{(strict)}\min_{[0,T)\times\mathbb{R}} (\vrmglk-\vp)=(\vrmglk-\vp)(t_0, x_0)=0.
$$
Recall that $\vrmglk$ is lower-semicontinuous by Proposition \ref{prop: existence faible}.

Because the infimum is achieved in the definition  of $\vrmglk$, by the afore-mentioned proposition, there exists  $|y_{0}|\le k$ and  $\tilde \P\in \tilde \Ac_{k}\cap \tilde \Gc_{k}(t_{0},x_{0},v_{0},y_{0})$,  such that $v_{0}:=c_{0}+y_{0}x_{0}=\vrmglk(t_{0},x_{0})$ for some $c_{0}\in \R$.  Let us set $(\tilde X,\tilde Y,\tilde V):=(\tilde X^{x_{0},\tilde \phi},\tilde Y^{\tilde \phi}, \tilde V^{x_{0},v_{0},\tilde \phi})$ where $\tilde \phi=(y_{0},\tilde a,\tilde b)$.    Let $\theta_{o}$ be a stopping time for the augmentation of the raw filtration $\tilde {\mathbb F}^{\circ}$, and define
$$
\theta :=\theta_{o} \wedge \theta_{1} \mbox{ with } \theta_{1}:=\inf\{s: |\tilde{X}_s-x_0| \ge 1\}.
$$
 Then, it follows from Proposition \ref{prop: gdp1} below that
$$
\tilde{V}_{\theta_{o}} \ge \vrmglk(t_{0}+\theta_{o}, \tilde{X}_{\theta_o}) \ge \vp(t_{0}+\theta_o, \tilde{X}_{\theta_o}),
$$
 {in which here and hereafter inequalities are taken in the $\tilde \P$-a.s.~sense.}
After applying It\^{o}'s formula twice, the above inequality reads:
\begin{equation}\label{eq: deb sur sol inega at theta eta}
\int_{0}^{{\theta}}\ell_{s}\,ds+\int_{0}^{{\theta}}\left(y_0-\partial_{x}\vp(t_{0},x_{0})+\int_{0}^{s} m_{r} dr +\int_{0}^{s} n_{r} d\tilde X_{r}\right)d\tilde X_{s}  \ge 0.
\end{equation}
where
\b*
&\ell  := \frac 12 \tilde a^{2}f(\tilde X)-\mathcal{L}^{\tilde a}\vp(t_{0}+\cdot, \tilde{X}_\cdot) \;,\;
m  := \mu_{Y}^{\tilde a,\tilde b}(\tilde X)-\mathcal{L}^{\tilde a}\partial_{x}\vp(t_{0}+\cdot, \tilde{X}_\cdot)& \\
&n :=  \gamma^{\tilde a}_{Y}(\tilde X)-\partial^{2}_{xx}\vp(t_{0}+\cdot, \tilde{X}_\cdot), &
\e*
with
$$
\mathcal{L}^{\tilde a}:=\partial_{t} +\frac12 (\sigma_{X}^{\tilde a})^{2} \partial^{2}_{xx}
$$
For the rest of the proof, we recall  \reff{eq: cond weak 2}.
Together with \reff{eq: H1} and Remark \ref{rem: borne a et sigma a},  this implies that $\sigma_{X}^{\tilde a}(\tilde X)$, $ \sigma_{X}^{\tilde a}(\tilde X)^{-1}$ and $\mu_{X}^{\tilde a,\tilde b}(\tilde X)$ are $\tilde \P$-essentially  bounded. After performing an equivalent  change of measure, we can thus find  $\check \P\sim \tilde \P$ and a $\check \P$-Brownian motion $\check W$ such that:
\begin{equation}
\tilde X=\int_{0}^{\cdot} \sigma_{X}^{\tilde a_{s}}(\tilde X_{s})d\check W_{s}.
\end{equation}
Clearly, both $\check \P$ and $\check W$ depend on $(\tilde a,\tilde b,y_{0})$.

\noindent {\bf 1.} We first show  that $y_{0}=\partial_{x}\vp(t_{0},x_{0})$, and therefore
\begin{equation}\label{eq: deb sur sol inega at theta eta sans y-vx}
\int_{0}^{{\theta}}\ell_{s}\,ds+{\int_{0}^{{\theta}}\int_{0}^{s} m_{r} dr d\tilde X_{s}}+ \int_{0}^{{\theta}} \int_{0}^{s} n_{r} d\tilde X_{r}d\tilde X_{s}  \ge 0.
\end{equation}
 Let $\check \P^{\lambda}\sim \check \P$ be the measure under which
$$
\check W^{\lambda}:=\check W+\int_{0}^{\cdot }\lambda{[\sigma_{X}^{\tilde a_{s}}(\tilde X_{s})]^{-1}}(y_{0}-\partial_{x}\vp(t_{0},x_{0}))ds
$$
is a $\check \P^{\lambda}$-Brownian motion. Consider the case $\theta_{o}:=\eta>0$. Since all the coefficients are bounded, taking expectation under $\check \P^{\lambda}$ and using \reff{eq: deb sur sol inega at theta eta} imply
\begin{align*}
C'\eta&\ge \lambda (y_{0}-\partial_{x}\vp(t_{0},x_{0}))^{2}\E^{\check \P^{\lambda}}\left[\theta\right]
\\
&{+\E^{\check \P^{\lambda}}\left[ \int_{0}^{{\theta}}\left( \int_{0}^{s} m_{r} dr +\int_{0}^{s} n_{r} d\tilde X_{r}\right)\lambda (y_{0}-\partial_{x}\vp(t_{0},x_{0}))ds \right]}
\end{align*}
for some $C'>0$. We now divide both sides by $\eta$ and use the fact that $(\eta \wedge \theta_{1})/\eta \to 1$ $\check \P^{\lambda}$-a.s.~as $\eta\to 0$ {to obtain
\begin{align*}
C'&\ge \lambda (y_{0}-\partial_{x}\vp(t_{0},x_{0}))^{2}.
\end{align*}
Then, we send} $\lambda \to \infty$ to deduce that $y_{0}=\partial_{x}\vp(t_{0},x_{0})$.

\noindent {\bf 2.} We now prove that
\begin{equation}\label{eq: born gamma super sol}
\partial_{xx}^{2}\vp(t_{0},x_{0})\le \gamma^{\tilde a_{0}}_{Y}(x_{0})\le \bar \gamma(x_{0}).
\end{equation}
We first consider the  time change
$$
h(t)=\inf\{r\ge 0: \int_{0}^{r} \left[(\sigma^{\tilde a_{s}}_{X}(\tilde X_{s}))^{2}\1_{[{0},\theta]}(s) +\1_{[{0},\theta]^{c}}(s) \right] ds\ge t\}.
$$
Again, $\sigma_{X}^{\tilde a}(\tilde X)$ and $ \sigma_{X}^{\tilde a}(\tilde X)^{-1}$ are essentially bounded by Remark \ref{rem: borne a et sigma a},  so that $h$ is absolutely continuous and its density ${\mathfrak h}$ satisfies
\begin{equation}\label{eq: contrôle changement de temps}
0<\underline {\mathfrak h} t  \le {\mathfrak h}(t):=\left[(\sigma^{\tilde a}_{X}(\tilde X))^{2}\1_{[{0},\theta]}(t)+\1_{[{0},\theta]^{c}}(t) \right]^{-1} \le \bar {\mathfrak h}t
\end{equation}
for some constants $\underline {\mathfrak h}$ and $\bar {\mathfrak h}$, for all $t\ge 0$. Moreover,  $\hat W:= \tilde X_{h}$ is a Brownian motion in the time changed filtration. Let us now take $\theta_{o}:=h^{-1}(\eta)$ for some $0<\eta<1$. Then, \reff{eq: deb sur sol inega at theta eta sans y-vx} reads
\begin{align}\label{eq: deb sur sol inega at theta eta sans y-vx time changed}
0\le &\int_{0}^{{\eta\wedge h^{-1}(\theta_{1})}}\ell_{h(s)}{\mathfrak h}(s)\,ds+{\int_{0}^{\eta\wedge h^{-1}(\theta_{1})} \int_{0}^{s} m_{h(r)} {\mathfrak h}(r) dr d\hat W_{s}}\nonumber\\
&+\int_{0}^{\eta\wedge h^{-1}(\theta_{1})} \int_{0}^{s} n_{h(r)} d\hat W_{r}d\hat W_{s}.
\end{align}
{Since all the involved processes are continuous and bounded, and since $( \eta \wedge h^{-1}(\theta_{1}))/\eta\to 1$ a.s.~as $\eta\to 0$, the above combined with \cite[Theorem A.1 b.~and Proposition A.3]{cheridito2005multi} } implies that
$$
\gamma^{\tilde a_{0}}_{Y}(x_{0})-\partial^{2}_{xx}\vp(t_{0}, x_{0})=\lim_{r\downarrow {0}} n_{h(r)} =\lim_{r\downarrow {0}} n_{r}\ge 0.
$$
Since $\gamma^{\tilde a}_{Y}(\tilde X)\le \bar \gamma(\tilde X)$, this proves \reff{eq: born gamma super sol}.

\noindent {\bf 3.} It remains to show that the first term in the definition of $F[\vp](t_{0},x_{0})$ is also non-negative, recall \reff{eq: pde contrainte interior}.
Again, let us take $\theta_{o}:=h^{-1}(\eta)$ and recall from 2.~that $\lim_{\eta\to 0}(\eta\wedge h^{-1}(\theta_{1}))/\eta=1$ $\check \P$-a.s. Note that $\tilde a$ being of the form \reff{eq: cond weak 1} with the condition \reff{eq: cond weak 2}, it satisfies \cite[Condition (A.2)]{cheridito2005multi}{, and so does $n$}.  Using {\cite[Theorem A.2 and Proposition A.3]{cheridito2005multi}} and \reff{eq: deb sur sol inega at theta eta sans y-vx time changed}, we then deduce that
$
\ell_{{0}}{\mathfrak h}({0})-\frac12 n_{{0}}\ge 0.
$
Hence, \reff{eq: contrôle changement de temps} and direct computations {based on \reff{eq: def gamma a}} imply
\begin{align*}
0&\le  \frac 12 {\tilde a}^{2}_{{0}}{f(x_{0})}-\mathcal{L}^{\tilde a_{{0}}}\vp(t_{0}, x_{0})-\frac12\left(\gamma^{\tilde a_{{0}}}_{Y}(x_{0})-\partial^{2}_{xx}\vp(t_{0}, x_{0})\right)(\sigma_{X}^{\tilde a_{{0}}}(x_{0}))^{2}\\
&=  \frac 12 \tilde a^{2}_{{0}}{f(x_{0})}-\partial_{t}\vp(t_{0}, x_{0})-\frac12 \gamma^{\tilde a_{{0}}}_{Y}(x_{0})(\sigma_{X}^{\tilde a_{{0}}}(x_{0}))^{2}
\\
&=-\partial_{t}\vp(t_{0}, x_{0})-\frac12 \frac{\sigma^{2}(x_{0})}{1-f(x_{0})\gamma^{\tilde a_{{0}}}_{Y}(x_{0}) } \gamma^{\tilde a_{{0}}}_{Y}(x_{0})
\\
&\le -\partial_{t}\vp(t_{0}, x_{0})-\frac12 \frac{\sigma^{2}(x_{0})}{1-f(x_{0}) \partial^{2}_{xx}\vp(t_{0}, x_{0})}  \partial^{2}_{xx}\vp(t_{0}, x_{0}),
\end{align*}
in which we use the facts that $\partial^{2}_{xx}\vp(t_{0}, x_{0})\le \gamma^{\tilde a_{{0}}}_{Y}(x_{0}) \le \bar \gamma(x_{0})$ and   $z\mapsto z/(1-f(x_{0})z)$ in non-decreasing on $(-\infty,\bar \gamma(x_{0})]\subset (-\infty,1/f(x_{0}))$, for the last inequality.

\noindent{\bf b.} We now consider the boundary condition at $T$. Since $\vrmglk$ is a supersolution of $\bar \gamma-\partial^{2}_{xx}\vp\ge 0$ on $[0,T)\x \R$, the same arguments as in \cite[Lemma 5.1]{CVT} imply that $\vrmglk-\bar \Gamma$ is concave for any twice differentiable function $\bar \Gamma$ such that $\partial_{xx}^{2} \bar \Gamma=\bar \gamma$.
The function $\vrmglk$ being lower-semicontinuous, the map
$$
x\mapsto G(x):=\liminf_{\tiny\begin{array}{c}t'\to T,x'\to x\\t'<T\end{array}} \vrmglk(t',x')
$$
is such that $G\ge g$ and $G-\bar \Gamma$ is concave. Hence,
 $G=(G-\bar \Gamma)^{\rm conc}+\bar \Gamma$ $\ge$ $(g-\bar \Gamma)^{\rm conc}+\bar \Gamma$ $=\hat g$.
\ep
\vs2

It remains to state the dynamic programming principle used in the above proof.

\begin{proposition}\label{prop: gdp1}
Fix $(t,x,v,y)\in [0,T]\x \R^{2}\x [-k,k]$ and let $\theta$ be a stopping time for the $\tilde \P$-augmentation of $\tilde {\mathbb F}^{\circ}$ that takes  $\tilde \P$-a.s.~values in $[0,T-t]$.
Assume that $\tilde \P\in   \tilde \Ac_{k}\cap \tilde \Gc_{k,\bar \gamma}(t,x,v,y)$. Then,
\b*
\tilde  V^{x,v,\tilde \phi}_{\theta}\ge \vrmglk(t+\theta, \tilde X^{x,\tilde \phi}_{\theta})\;\;\tilde \P-{\rm a.s.},
 \e*
 in which $\tilde \phi:=(y,\tilde a,\tilde b)$.
\end{proposition}

 \proof  Since $\vrmglk$ is lower-semicontinuous and all the involved processes have continuous paths, up to  approximating $\theta$ by a sequence of stopping times valued in finite time grids,  it suffices to prove our claim in the case $\theta\equiv r \in [0,T-t]$.   Let $\tilde \P_{\omega}$ be a regular conditional probability given $\tilde \Fc_{r}^{\circ}$ for $\tilde \P$. It coincides with $\tilde \P[\cdot|\tilde \Fc_{r}^{\circ}](\omega)$ outside a set $N$ of $\tilde \P$-measure zero.
 Then,   for all $\omega\notin N$,  $0\le \delta\le 1$ and $r\ge 0$ the conditions \reff{eq: cond weak 1}-\reff{eq: cond weak 2}-\reff{eq: cond weak 3} hold for  $\tilde \P^{r}_{\omega}$  defined on $C(\R_{+})^{5}$ by
 $$
  \tilde \P^{r}_{\omega}[\omega'\in A]=\tilde \P_{\omega}[\omega'_{r+\cdot} \in A].
 $$
 Moreover, \cite[Theorem 3.3]{claisse2014} ensures that, after possibly modifying $N$,
 \b*
 &&\tilde \P^{r}_{\omega}\left[\tilde V^{\xi_{r}(\omega),\vartheta_{r}(\omega),\hat \phi(\omega)}_{T-(t+r)}\ge g(\tilde X^{\xi_{r}(\omega),\hat \phi(\omega) }_{T-(t+r)})\right]=1\;\\
 &&\mbox{ and }\;
 \tilde \P^{r}_{\omega}\left[ \gamma^{\tilde a}_{Y}(\tilde X^{\xi_{r}(\omega),\hat \phi(\omega) })\le \bar \gamma (\tilde X^{\xi_{r}(\omega),\hat \phi(\omega) })\mbox{ on } \R_{+}\right]=1,
 \e*
 for $\omega \notin N$,
 in which
 $$
 (\xi_{r},\vartheta_{r},\hat \phi):=(\tilde X^{x,\tilde \phi }_{r},\tilde  V^{x,v,\tilde \phi}_{r},(\tilde  Y^{x,\tilde \phi}_{r},\tilde a,\tilde b)).
 $$
 This shows that $\vartheta_{r}(\omega)\ge \vrmglk(t+r,\xi_{r}(\omega))$ outside the null set $N$, which is the required result.
 \ep

\subsection{Conclusion of the proof and construction of almost optimal strategies}

We first conclude the proof of Theorem \ref{thm: main}.

\noindent\textbf{Proof of Theorem \ref{thm: main}.}   Proposition \ref{prop: conv bar v eps to bar v} and  Theorem \ref{thm:Reg} imply that $\vrmgh\ge \vrmg$ in which $\vrmgh$ has linear growth and is a continuous viscosity solution of  \reff{eq: VisEqBd}. On the other hand, Proposition \ref {prop: v weak is lower bound} and  Theorem \ref {thm: weak v is supersolution} imply that $\vrmgl\le \vrmg$ on $[0,T)\x \R$ in which $\vrmgl$ has linear growth and is a  viscosity supersolution of  \reff{eq: VisEqBd}. By the comparison result of Theorem \ref{thm: Comp} applied with $\epsilon=0$, $\vrmgl\ge \vrmgh$. Hence,
\be\label{eq: ega des v}
\vrmg=\vrmgl=\vrmgh\;\mbox{  on $[0,T)\x \R$ and } \vrmgl=\vrmgh \mbox{  on $[0,T]\x \R$}
\ee
Since $\vrmgh$ is continuous, this shows that  
$$
\lim_{\tiny \begin{array}{c}(t',x')\to (T,x)\\ t'<T\end{array}} \vrmg(t',x')=\vrmgh(T,x)=\vrmgl(T,x).
$$
Hence,   $\vrmg$ is a viscosity solution of  \reff{eq: VisEqBd}, with linear growth.
\ep

\begin{remark}[Almost optimal controls]\label{rem: almost optimal controls}  In  the  proof of Theorem \ref{thm:Reg}, we have constructed a super-hedging strategy starting from $\vrmgbed(t,x)$. Since  $\vrmgbed(t,x)\to \vrmgh(t,x)=\vrmg(t,x)$ as $\delta,\epsilon\to 0$ and $K \to \infty$, this provides a way to construct super-hedging strategies associated to any initial wealth $v> \vrmg(t,x)$.
\end{remark}

\section{Adding a resilience effect}\label{sec: resilience}

In this section, we explain how a resilience effect can be added to our model. In the discrete rebalancement setting, we replace the dynamics \reff{eq: dyna Sn} by
\b*
\Sn  = X_{0} + \int_{0}^{\cdot}\mu(\Sn_{s})ds +\int_{0}^{\cdot} \sigma(\Sn_{s}) dW_{s} + R^{{n}},
\e*
in which $R^{{n}}$ is defined by
$$
R^{{n}}=R_{0}+ \sum_{i=1}^{n}\1_{[\ti,T]} \deltan_{\ti} f(\Sn_{\ti-}) -\int_{0}^{\cdot} \rho R^{{n}}_{s}ds,
$$
for some $\rho>0$ and $R_{0}\in \R$.
The process $R^{{n}}$ models the impact of past trades on the price, the last term in its dynamics is the resilience effect. Then, the continuous time dynamics becomes
\begin{align*}
X&=X_{0 }+\int_{0}^{\cdot} \sigma(X_{s}) dW_{s} + \int_{0}^{\cdot} f(X_{s}) dY_{s}+\int_{0}^{\cdot} (\mu(X_{s})+a_{s}(\sigma f')(X_{s})-\rho R_{s} )ds  \\
R&=R_{0}+ \int_{0}^{\cdot} f(X_{s}) dY_{s}+\int_{0}^{\cdot} (a_{s}(\sigma f')(X_{s})-\rho R_{s} )ds\\
  V&=V_{0 }+ \int_{0}^{\cdot} Y_{s}dX_{s}+\frac12 \int_{0}^{\cdot} a^{2}_{s} f(X_{s}) ds.
\end{align*}
This is obtained as a straightforward extension of \cite[Proposition 1.1]{bouchard2015almost}.
\def\wrmg{{\rm v}^{R}_{\bar \gamma}}
\vs2

Let $\wrmg(t,x)$ be defined as the super-hedging price $\vrmg(t,x)$ but for these new dynamics and for $R_{t}=0$. The following states that $\wrmg=\vrmg$, i.e.~adding a resilience effect does not affect the super-hedging price.
\vs2

\begin{proposition}\label{prop: resilience}   $\vrmg= \wrmg$ on $[0,T]\x \R$.
\end{proposition}

\proof {\bf 1.} To show that $\vrmg\ge  \wrmg$,  it suffices to reproduce the arguments of the proof of Theorem \ref{thm:Reg} in which the drift part of the dynamics of $X$ does not play any role. More precisely, these arguments show that  ${\vrmgh}\ge  \wrmg$. Then, one uses the fact that $\vrmg={\vrmgh}$, by \reff{eq: ega des v}. \\
{\bf 2.} As for the opposite inequality, we use the weak formulation of Section  \ref{sec: weak form} and a simple Girsanov's transformation.
 For ease of notations, we restrict to $t=0$. Fix $v>\wrmg(0,x)$, for some $x\in \R$. Then, one can find $k\ge 1$,  $(c,y)\in \R\x [-k,k]$ satisfying $v=c+yx$, and $(a,b)\in \Ac_{k,\bar \gamma}(0,x) $ such that $V_{T}\ge g(X_{T})$,  with $(V,X,Y,R)$ defined by the corresponding initial data and controls. We let
$$
a=a_{0}+\int_{0}^{\cdot}\beta_{s} ds +\int_{0}^{\cdot} \alpha_{s} dW_{s}
$$
be the decomposition of $a$ into an It\^{o} process, see Section \ref{sec: impact rule and discrete time trading}.
 Let $\Q^{R}\sim \P$ be the probability measure under which $W^{R}:=W-\int_{0}^{\cdot} (\rho{R_{s}}/\sigma(X_{s}))ds$ is a $\Q^{R}$-Brownian motion, recall \reff{eq: H1}.
Then,
\begin{align*}
X&=X_{0 }+\int_{0}^{\cdot} \sigma(X_{s}) dW^{R}_{s} + \int_{0}^{\cdot} f(X_{s}) dY_{s}+\int_{0}^{\cdot} (\mu(X_{s})+a_{s}(\sigma f')(X_{s}) )ds  \\
Y&=Y_{0}+\int_{0}^{\cdot} (b_{s}+a_{s}\rho {R_{s}}/\sigma(X_{s}) )ds+ \int_{0}^{\cdot} a_{s}dW^{R}_{s}\\
a&=a_{0}+\int_{0}^{\cdot}(\beta_{s}+ \alpha_{s}\rho {R_{s}}/\sigma(X_{s}))ds +\int_{0}^{\cdot} \alpha_{s} dW^{R}_{s}\\
  V&=V_{0 }+ \int_{0}^{\cdot} Y_{s}dX_{s}+\frac12 \int_{0}^{\cdot} a^{2}_{s} f(X_{s}) ds.
\end{align*}
Upon seeing $(a,b+a\rho {R }/\sigma(X),\alpha,\beta+ \alpha\rho{R }/\sigma(X),W^{R} )$ as a generic element of the canonical space $C([0,T])^{5}$ introduced in Section  \ref{sec: weak form}, then $\Q^{R}$ belongs to  $\tilde \Ac_{k}\cap \tilde \Gc_{k,\bar \gamma}(t,x,v,y)$, and therefore $v>\vrmgl(0,x)$. Hence, $\wrmg(0,x)\ge  \vrmgl(0,x)$, and thus  $\wrmg(0,x)\ge  \vrmg(0,x)$ by  \reff{eq: ega des v}.
\ep

\section{Numerical approximation and examples}

In this section, we  provide an example of  numerical schemes that converges towards the {unique} continuous viscosity solution of \reff{eq: VisEqBd} with linear growth. The scheme is then exemplified using two numerical applications in the case of constant market impact and gamma constraint.

\subsection{Finite difference scheme}\label{sec: finite diff scheme}
Given a map $\phi$ and $h:=(h_{t},h_{x})\in (0,1)^{2}$, define
\begin{align*}
L^{h}_{1}(t,x,y,\phi)&:=-\frac{\phi(t+h_{t},x)-y}{h_{t}}-\frac{\sigma^{2}(x)  G^{h}(t,x,y,\phi)}{2(1-f(x)G^{h}(t,x,y,\phi))}\\
L^{h}_{2}(t,x,y,\phi)&:=\bar \gamma(x)-G^{h}(t,x,y,\phi)
\end{align*}
where
$$
G^{h}(t,x,y,\phi):=\frac{\phi(t+h_{t},x+h_{x})+\phi(t+h_{t},x-h_{x})-2y}{h_{x}^{2}}.
$$

The numerical scheme is set on the grid $\pi_{h}:=\{(t_{i},x_{j})=(ih_t, \underline x+jh_x): i\le n_{t},j\le n_{x}\}$, with $n_{t}h_t=T$ for some $n_{t}\in \mathbb{N}$, and $n_{x}h_{x}=\overline x-\underline x$, for some real numbers $\underline x<\overline x$. To paraphrase,   $\vrmg^{h}$ is defined on $\pi_{h}$ as the solution of
\be\label{eq: numscheme}
S(h,t_i,x_j,\vrmg^{h}(t_{i},x_{j}),\vrmg^{h}) &=&0 \;\mbox{ for }~ i<n_{t}, 1\le j\le n_{x}-1 \\
\vrmg^{h}&=&\hat{g}\;\mbox{ on } \pi_{h}\cap \{(\{T\}\x \R)\cup( [0,T]\cap \{\underline x,\overline x\})\} \nonumber
\ee
where
$$
S(h,t,x,y,\phi) := (\bar w-y)\vee(y-\underline w)\wedge \min_{l=1,2}\left\{L^{h}_{l}(t,x,y,\phi)\right\}
$$
with $\bar w$ and  $\underline w$ as in Remark \ref{rem: bound en terme w bar underline}.

\begin{theorem}
The equation   \reff{eq: numscheme} admits a unique solution $\vrmg^{h}$, for all $h:=(h_t,h_x)\in (0,1)^{2}$. Moreover, if  $h_t/h_x^{2}\to 0$ and $h_{x}^{2}\to 0$, then $\vrmg^{h}$ converges locally uniformly to the unique   continuous viscosity solution of \reff{eq: VisEqBd} that has linear growth.
\end{theorem}
\begin{proof}
The existence of a   solution, that is bounded  by the map with linear growth $|\bar w|+|\underline w|$, is obvious. We now prove uniqueness. First observe that $L^h_2$ is strictly increasing in its $y$-component, and  that
$$
\frac{\partial L^{h}_{1}}{\partial y}(t,x,y,\phi)=\frac{1}{h_t}+\frac{\sigma^{2}(x)}{h^2_x (1-f(x)G^{h}(t,x,y,\phi))^{2}}>0
$$
on the domain $\{y: L^{h}_{2}(t_{i},x_{j},y,\phi) \ge 0\}$.
Uniqueness of the solution follows.

It is easy to see that $\phi\mapsto S(\cdot,\phi)$ is non-decreasing, so that our scheme is monotone. Consistency is clear.
Moreover, it is not difficult to check that the comparison result of Theorem \ref{thm: Comp}  extends to this equation (there is an equivalence of the notions of super- and subsolutions in the class of functions $w$ such that $\underline w \le w \le \bar w$). It then follows from  \cite[Theorem 2.1]{BaSo91} that $\vrmg^{h}$ converges locally uniformly to the unique continuous viscosity solution with linear growth of
$$\Big[ (\bar w-\vp)\vee(\vp-\underline w)\wedge F[{\vp}]\Big] \1_{[0,T)}+({\vp}-\hat{g})\1_{\{T\}}=0.$$
 In view of \reff{eq: ega des v}, Remark \ref{rem: bound en terme w bar underline} and Theorem \ref{thm: main},  $\vrmg$ is the unique viscosity solution  of the above equation.
\ep
\end{proof}

\subsection{Numerical examples: the fixed impact case}\label{sec: exemple num}

To illustrate the above numerical scheme, we place ourselves in the simpler case where $f\equiv \lambda>0$  and $\bar \gamma>0$ are constant. The dynamics of the stock is given by the Bachelier model
$$\,dX_t =\sigma \,dW_t,$$
with $\sigma:=0.2$. In the following, $T=2$.

First, we consider a European Butterfly option with three strikes $K_1=-1<K_2=0<K_3=1$, where $K_1+1/(2\bar \gamma) \le K_2 \le K_3-1/(2\bar \gamma)$.  Its pay-off is
$$g(x)=(x-K_1)^{+}-2(x-K_2)^{+}+(x-K_3)^{+},$$
and the corresponding face-lifted function $\hat{g}$ can be computed explicitly:
\be
\hat{g}(x)&=&\frac{\bar\gamma}{2} (x-x_1^{-})^{2}\1_{[x^{-}_1,x^{+}_1)}+(x-K_1)\1_{[x_1^{+}, K_2)} \nonumber\\
&& +(x-K_1-2(x-K_2))\1_{[K_2, x^{-}_2)}\nonumber \\
& &+\left(\frac{\bar\gamma}{2} (x-x^{+}_2)^{2}+2K_2-(K_1+K_3)\right)\1_{[x^{-}_2,x^{+}_2)} \nonumber \\
& &+(2K_2-(K_1+K_3))\1_{[x^{+}_2, +\infty)}, \nonumber
\ee
where $x^{\pm}_{1}=K_1\pm 1/(2\bar \gamma)$ and $x^{\pm}_{2}=K_3\pm 1/(2\bar \gamma)$.

In Figure \ref{fig:Butterfly}, we separately show the {effect} of the gamma constraint and of the market impact. As observed in Remark \ref{rem: pde non decreasing}, the price is  non-decreasing with respect to the impact parameter $\lambda$ and bounded from below by the hedging price obtained in the model without impact nor gamma constraint. On the left and right tails of the curves, we observe the effect of the gamma constraint. It does not operate around $x=0$ where the gamma is non-positive. The effect of the market impact operates only in areas of high convexity (around $x=-1.5$ and $x=1.5$) or of high concavity (around $x=0$).
\begin{figure}[H]
    \centering
    \begin{minipage}{0.5\textwidth}
        \centering
        \includegraphics[scale=0.4]{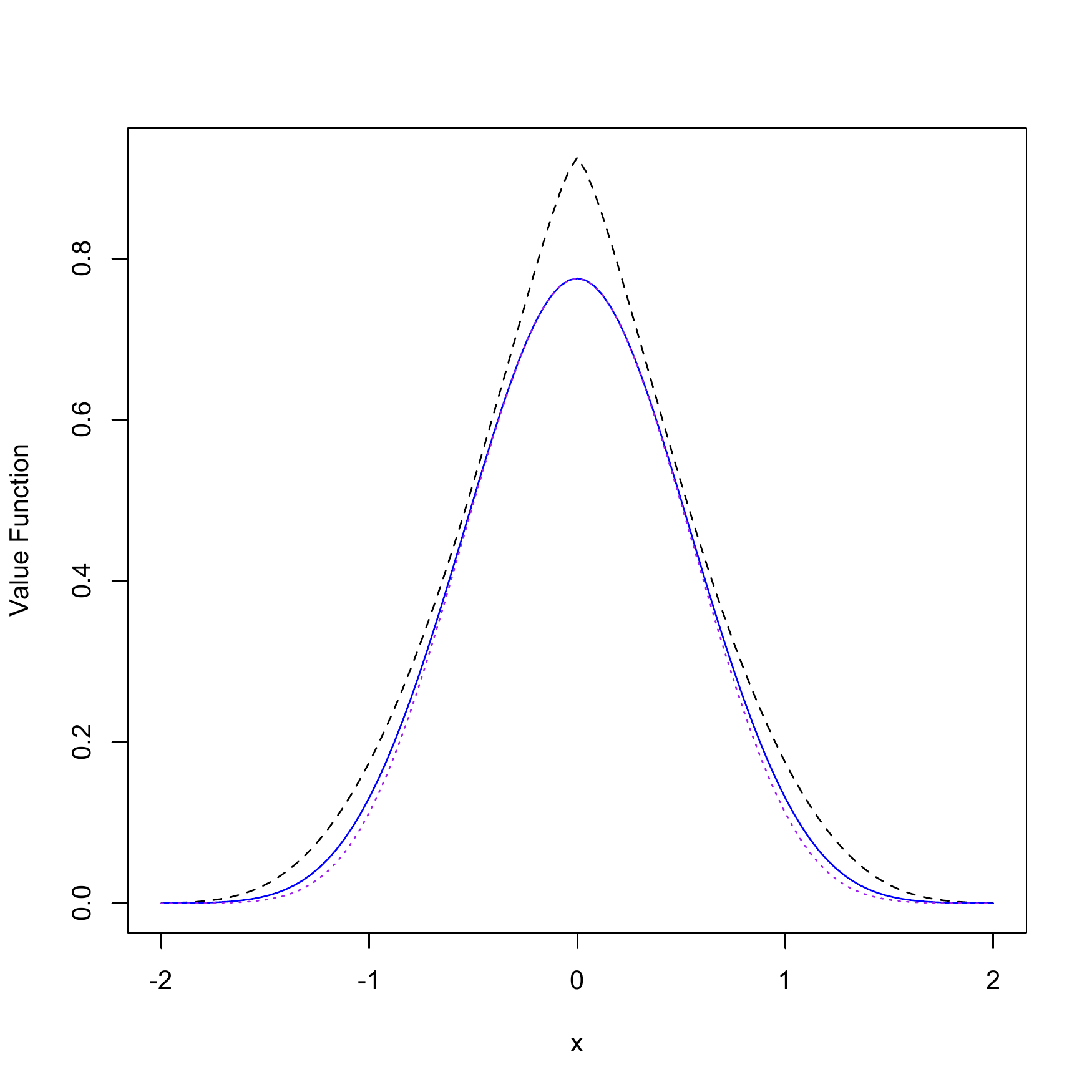}
    \end{minipage}%
    \begin{minipage}{0.5\textwidth}
        \centering
        \includegraphics[scale=0.4]{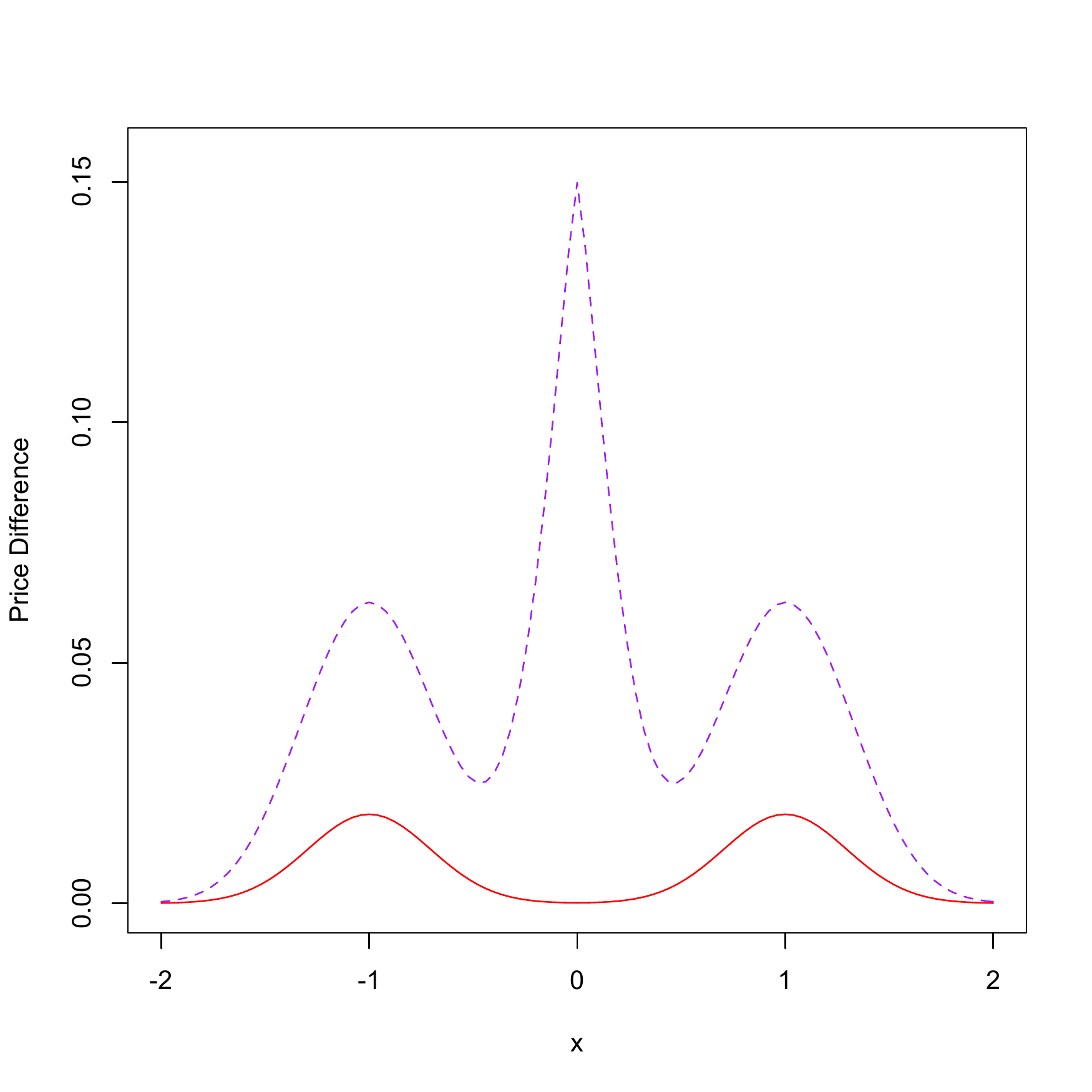}
    \end{minipage}
    \caption{Left: Super-hedging price of the Butterfly option. Dashed line: $\lambda=0.5$, $\bar \gamma=1.75$; solid line: $\lambda=0$, $\bar \gamma=1.75$; dotted line: $\lambda=0$, $\bar \gamma=+\infty$. Right: Difference with the price associated to $\lambda=0$, $\bar \gamma=+\infty$. Dashed line: $\lambda=0.5$, $\bar \gamma=1.75$; {solid} line: $\lambda=0$, $\bar \gamma=1.75$ .}
    \label{fig:Butterfly}
\end{figure}

In Figure \ref{fig:CallSpread}, we perform similar computations but for a call spread option, where 
$$g(x)=(x-K_1)^{+}-(x-K_2)^{+},$$
with $K_1=-1<K_2=1$ such that $K_1+1/(2\bar \gamma) \le K_2$. The face-lifted function $\hat{g}$ is given by
$$\hat{g}(x)=\frac{\bar\gamma}{2} (x-x^{-})^{2}\1_{[x^{-},x^{+})}+(x-K_1)\1_{[x^{+}, K_2)}+(K_2-K_1)\1_{[K_2, +\infty)}$$
with $x^{\pm}=K_1\pm1/(2\bar \gamma)$.
\begin{figure}[H]
    \centering
    \begin{minipage}{0.5\textwidth}
        \centering
        \includegraphics[scale=0.4]{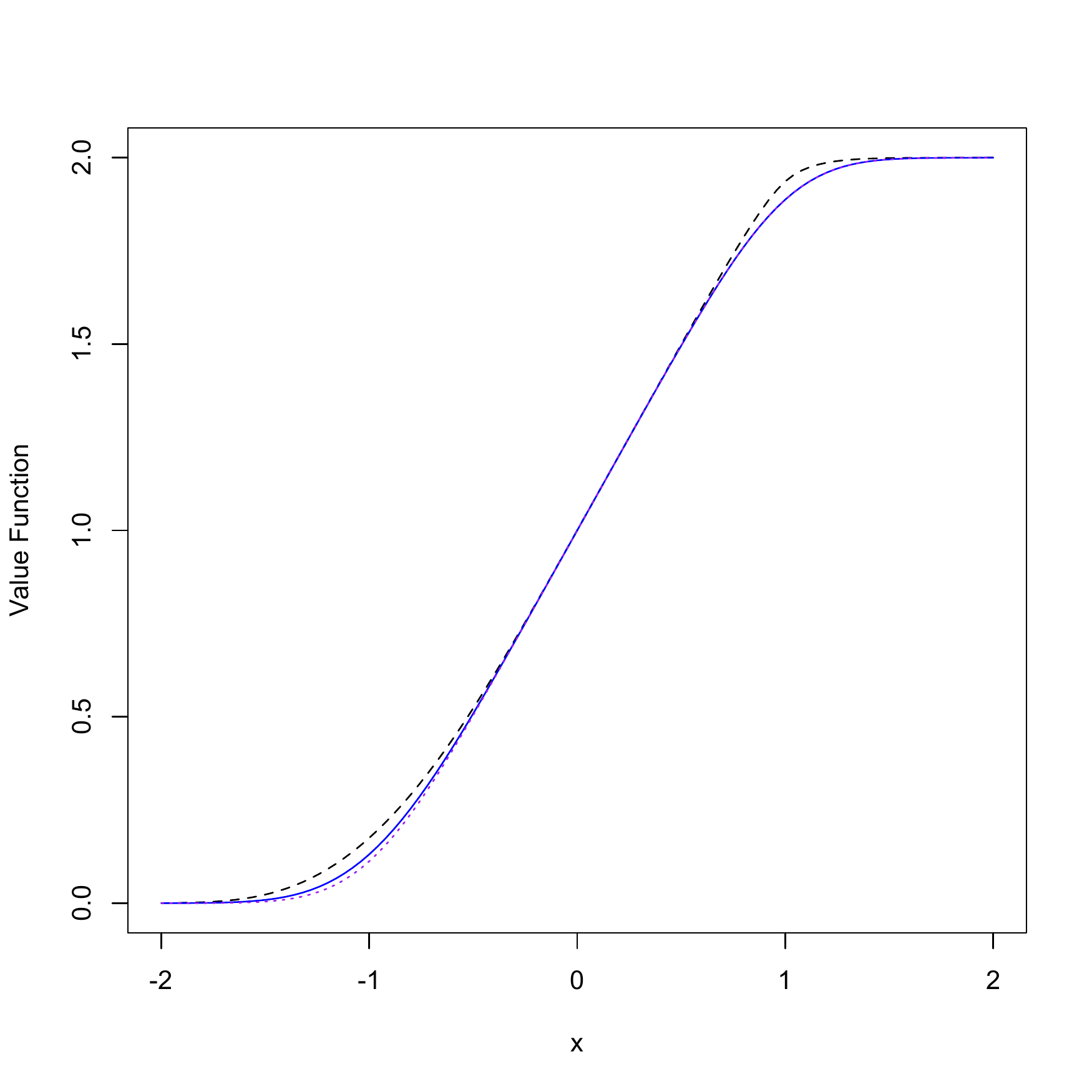}
    \end{minipage}%
    \begin{minipage}{0.5\textwidth}
        \centering
        \includegraphics[scale=0.4]{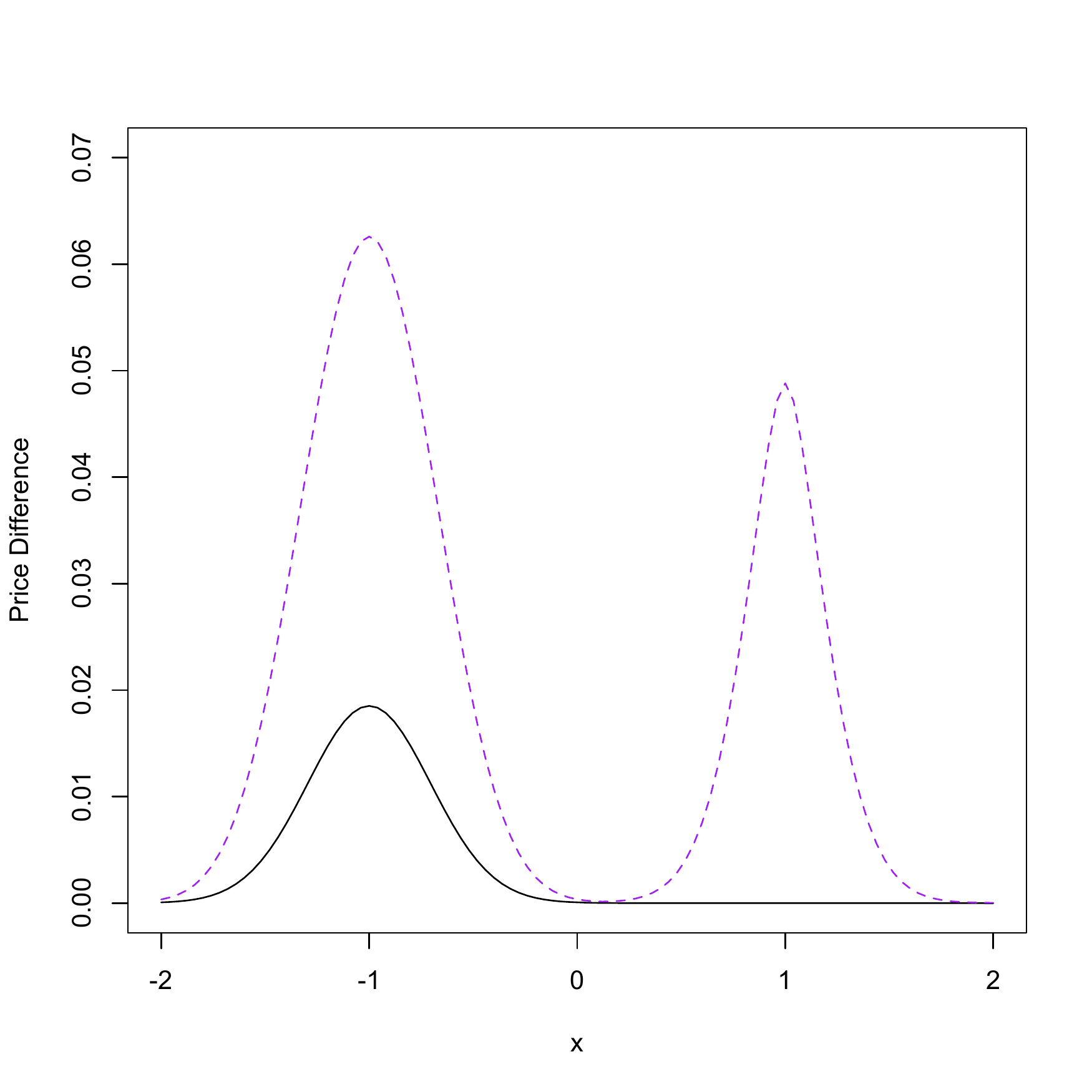}
    \end{minipage}
    \caption{Left: Super-hedging price of the Call Spread option. Dashed line: $\lambda=0.5$, $\bar \gamma=1.75$; solid line: $\lambda=0$, $\bar \gamma=1.75$; dotted line: $\lambda=0$, $\bar \gamma=+\infty$. Right: Difference with the price associated to $\lambda=0$, $\bar \gamma=+\infty$. Dashed line: $\lambda=0.5$, $\bar \gamma=1.75$; {solid} line: $\lambda=0$, $\bar \gamma=1.75$ .}
  \label{fig:CallSpread}
\end{figure}

%


\section{Appendix}

The following  is very standard, we prove it for completeness.

\begin{lemma}\label{lem: pertEquiv}
A upper-semicontinuous (resp.~lower-semicontinuous) map  is a viscosity subsolution (resp.~supersolution) of $$F_{\kappa}^{\epsilon}[\vp]\1_{[0,T)}+(\vp-\hat{g}_{K}^{\epsilon})\1_{\{T\}}=0$$
if and only if it is a viscosity subsolution (resp.~supersolution) of $F^{\epsilon,K}_{\kappa,-}[\vp]=0$ (resp. $F^{\epsilon,K}_{\kappa,+}[\vp]=0$).
\end{lemma}
\proof
The equivalence on $[0,T)$ is evident, we only consider the parabolic boundary $\{T\}\x \R$. Since $F^{\epsilon,K}_{\kappa,+}\ge F_{\kappa}^{\epsilon}$ and $F^{\epsilon,K}_{\kappa,-}\le F_{\kappa}^{\epsilon}$, only one implication is not completely trivial.

\noindent \textbf{a.}
Let $v$ be a viscosity supersolution of $F^{\epsilon,K}_{\kappa,+}[\vp]=0$, and $\vp \in C^{2}$ be a test function such that
$$\text{(strict)}\min_{[0,T]\times\mathbb{R}}(v-\vp)=(v-\vp)(T,x_0)=0,$$
for some $x_{0}\in \R$.
We define a new test function $\phi \in C^{2}$,
$$\phi(t,x) := \vp(t,x)-C(T-t),$$
so that $\partial_{t}\phi=\partial_{t}\vp+C$. For $C>0$ large enough,
$$\min_{x'\in D^{\epsilon}_{\kappa}}\min\left\{-\partial_{t}\phi-\frac{\sigma^{2}(x')\partial_{xx}\phi}{2(1-f(x')\partial_{xx}\phi)}, \bar \gamma(x')-\partial_{xx}\phi\right\} < 0$$
at $(T,x_{0})$. Since,
$$\text{(strict)}\min_{[0,T]\times\mathbb{R}}(v-\phi)=(v-\phi)(T,x_0)=0,$$
it must hold that  $F^{\epsilon,K}_{\kappa,+}[\phi](T,x_0) \ge 0$, and therefore
$$v(T,x_0)-\hat{g}_{K}^{\epsilon}(x_0)=\vp(T,x_0)-\hat{g}_{K}^{\epsilon}(x_0)=\phi(T,x_0)-\hat{g}_{K}^{\epsilon}(x_0) \ge 0.$$

\noindent \textbf{b.}
Let now $v$ be a viscosity subsolution of $F^{\epsilon,K}_{\kappa,-}[\vp]=0$, and $\vp \in C^{2}$ be a test function such that
$$\text{(strict)}\max_{[0,T]\times \mathbb{R}}(v-\vp)=(u-\vp)(T,x_0),$$
for some $x_{0}\in \R$.
Then,  $F^{\epsilon,K}_{\kappa,-}[\vp](T,x_0) \leq 0$. By replacing $\vp$ by ${\phi}$, defined for $\alpha>0$ as
$${\phi}(t,x):=\vp(t,x_0+\alpha (x-x_0))+C(T-t),$$
we obtain a new test function  at $(T, x_0)$. Since $\inf \bar \gamma>0$, recall \reff{eq: H1},  we can take  $\alpha$ small enough so that
$$\min\limits_{x'\in D^{\epsilon}_{\kappa}}\{\bar \gamma(x')-\partial_{xx}{\phi}(T,x_{0})\} > 0.$$
As in the previous step, we can now choose $C>0$ such that
$$\min_{x'\in D^{\epsilon}_{\kappa}}\left\{-\partial_{t}{\phi}-\frac{\sigma^{2}(x')\partial_{xx}{\phi}}{2(1-f(x')\partial_{xx}{\phi})}\right\} > 0
$$
at $(T,x_{0})$.
Since $F^{\epsilon,K}_{\kappa,-}[{\phi}](T,x_{0}) \leq 0$, we conclude that $v(T,x_{0})={\phi}(T,x_{0}) \leq \hat{g}_{K}^{\epsilon}(x_{0})$.
\ep


\nocite{*}
\bibliographystyle{plain}

\end{document}